\documentclass[12pt]{amsart}
\usepackage{amsmath,amssymb,amsthm,enumerate}
\usepackage[mathscr]{euscript}
\newcommand{\Res}{\mathrm{Res}}
\newcommand{\Ind}{\mathrm{Ind}}
\newcommand{\R}{\mathcal{R}}
\newcommand{\U}{\mathcal{U}}

\newcommand{\SL}{\mathrm{SL}}
\newcommand{\PGL}{\mathrm{PGL}}
\newcommand{\GL}{\mathrm{GL}}

\newcommand{\HH}{\mathcal{H}}

\newcommand{\G}{\mathcal{G}}
\newcommand{\Sres}{\mathcal{S}}
\newcommand{\GG}{\mathbb{G}}
\newcommand{\NN}{\mathbb{N}}
\newcommand{\BB}{\mathbb{B}}
\newcommand{\MM}{\mathbb{M}}

\newcommand{\lrc}[1]{\lceil #1 \rceil}
\newcommand{\lrf}[1]{\lfloor #1 \rfloor}

\newcommand{\T}{\mathcal{T}}

\newcommand{\Stor}{\mathbb{S}}
\newcommand{\ratk}{k}
\newcommand{\resk}{\kappa}

\newcommand{\extresk}{\kappa'}

\newcommand{\PP}{\mathcal{P}}
\newcommand{\p}{\varpi}
\newcommand{\val}{\mathrm{val}}
\newcommand{\real}{\mathbb{R}}
\newcommand{\Rplus}{\widetilde{\mathbb{R}}}
\newcommand{\A}{\mathscr{A}}
\newcommand{\build}{\mathscr{B}}

\newcommand{\cind}{\textrm{c-}\mathrm{Ind}}

\newcommand{\Hom}{\mathrm{Hom}}

\newcommand{\smat}[1]{\left[ \begin{smallmatrix} #1 \end{smallmatrix} \right]}

\newcommand{\lconj}[2]{{}^{#1}#2}      

\newcommand{\Para}{\mathcal{P}}
\newcommand{\B}{\mathcal{B}}

\theoremstyle{plain}
\newtheorem{theorem}{Theorem}[section]

\newtheorem{lemma}[theorem]{Lemma}
\newtheorem{proposition}[theorem]{Proposition}
\newtheorem{corollary}[theorem]{Corollary}

\theoremstyle{definition}
\newtheorem{definition}[theorem]{Definition}
\newtheorem{remark}[theorem]{Remark}

\theoremstyle{remark}
\newtheorem{example}{Example}

\numberwithin{equation}{section}

\begin{document}

\title[On depth-zero branching rules]{On Branching Rules of Depth-Zero Representations}
\author{Monica Nevins}
\address{Department of Mathematics and Statistics, University of Ottawa, Ottawa, Canada K1N 6N5}
\email{mnevins@uottawa.ca}
\thanks{This research is supported by a Discovery Grant from NSERC Canada.}
\subjclass{20G05}
\date{\today}

\begin{abstract}
Using Bruhat-Tits theory, we analyse the restriction of depth-zero representations of a semisimple simply connected $p$-adic group $G$ to a maximal compact subgroup $K$.  We prove the coincidence of branching rules within classes of Deligne-Lusztig supercuspidal representations.  Furthermore, we show that under obvious compatibility conditions, the restriction to $K$ of a Deligne-Lusztig supercuspidal representation of $G$ intertwines with the restriction of a depth-zero principal series representation in infinitely many distinct components of arbitrarily large depth.  Several qualitative and quantitative results are obtained, and their use is illustrated in an example.

\end{abstract}

\maketitle

\setlength{\parskip}{6pt}

\section{Introduction}

The branching rules considered here are those arising from the restriction of a complex admissible representation of a $p$-adic group $G$ to a maximal compact open subgroup $K$.  The ultimate goal of this analysis is to examine the interplay between the admissible duals of $G$ and $K$, as well as to illuminate their respective structures.  Aspects of this question for $G$ include the theory of types and the study of newforms.  On the other hand, the representation theory of $K$ is still in its infancy, and branching rules provide a framework in which to search for results. 

In this paper, we consider the restriction of certain depth-zero supercuspidal representations (those induced from inflations of Deligne-Lusztig cuspidal representations of associated finite groups of Lie type) to a hyperspecial maximal compact subgroup (denoted $G_y$) under the hypothesis that $G$ is connected, simply connected, semisimple and split over a local non-archimedean ground field $\ratk$ of odd residual characteristic.   
Our particular focus is the set of representations of $G_y$ which are common to the branching rules of several different representations of $G$.  We may call these components ``atypical'' to distinguish them from the more special ``types,'' in the sense of Bushnell-Kutzko or Moy-Prasad, that in some cases uniquely identify the irreducible representation of $G$ which contain them.    To this end we prove two main results.

The first concerns Deligne-Lusztig supercuspidal representations (recalled in Section~\ref{S:DL}).  We parametrize their coarse decomposition into Mackey components by a set $X_{x,y}^+$ in Sections~\ref{S:doublecosets} and \ref{S:restriction}.  In Theorem~\ref{T:same} we prove that whenever two Deligne-Lusztig supercuspidal representations arise from the same minisotropic torus and have the same central character, then a large portion of their branching rules are identical, namely, those parametrized by $int(X_{x,y}^+)$.  Moreover, we show that in certain circumstances the entirety of their restrictions to $G_y$ are identical (Corollary~\ref{C:allsame}).

The second main result, stated in Theorem~\ref{T:main}, concerns the intertwining between restrictions of Deligne-Lusztig supercuspidal representations and principal series.  We prove that the restriction of a Deligne-Lusztig supercuspidal representation $\pi$ intertwines in infinitely many distinct components with any compatible depth-zero principal series representation $\Ind_B^G\chi$.  The compatibility condition relates to the central character of the cuspidal representation inducing to $\pi$.   Further refinements of this result, relating to the depths at which these intertwinings occur, are given as a sequence of corollaries in Section~\ref{S:intertwining}.

One of the main methods underlying the proofs of these results, beyond Mackey theory, is the analysis of subgroups of $G$ which are stabilizers of subsets of an apartment $\A$.  A result of independent interest is Proposition~\ref{P:gxromega}, where we relate certain stabilizer subgroups with Moy-Prasad filtration subgroups.  We use this in Theorem~\ref{T:depth} and Proposition~\ref{P:depthps} to glean information about the depths of the representations of $G_y$ which arise.

Proving results on branching rules at this level of generality is a new and novel step, and anticipates the development of a general theory out of the case-by-case analysis achieved to date.  In this sense the current work complements a series by the author on branching rules of $\SL(2,k)$ \cite{Nevins2005, Nevins2011, Nevins2013} and, with P.~Campbell, $\GL(3,k)$ \cite{CampbellNevins2009,CampbellNevins2010}.  Recently, U.~Onn and P.~Singla in \cite{OnnSingla2014} determined the complete decomposition into irreducible representations of the blocks of representations of \cite{CampbellNevins2009}.  We use their results in our example in Section~\ref{S:example} and anticipate that in fact the complete branching rules for depth-zero representations of $\GL(3,k)$ are now attainable, using the above results and ideas inspired from the present paper.  

More generally, G.~Savin determined the branching rules of the minimal representation of split simply connected groups of types $D_n$ ($n\geq 4$) and $E_n$ ($6\leq n\leq 8$) in \cite{Savin1996}.  For these representations $V$, each $V^{G_{y,i+1}}/V^{G_{y,i}}$ is irreducible, similar to the case of (all irreducible representations of) $GL(2,\ratk)$ and $\PGL(2,\ratk)$; \cite{Savin1996} makes clever use of the Harish-Chandra--Howe character formula to identify and explicitly construct these irreducible components.

We assume that $G$ is semisimple; this simplifies the exposition, particularly in Section~\ref{S:background}.  It is feasible and would be interesting to extend the results to $G$ reductive, so that the Levi components of the proper parabolic subgroups are also in this class.  This could allow an inductive analysis of branching rules including all parabolically induced representations.

Our proofs of the main results Theorem~\ref{T:same} and  Theorem~\ref{T:main} rely on showing that certain double cosets support nonzero intertwining operators.  These questions reduce to computations with Deligne-Lusztig characters.
To determine which other double cosets also support intertwining operators would seem to require restricting representations to subgroups which are stabilizers of subsets of the Bruhat-Tits building not contained in any single apartment, and there is currently a dearth of literature on such subgroups.  Moreover, the classification of the double coset spaces which arise is expected to be highly nontrivial: for $\GL(n,k)$, $n \geq 3$, it was shown by  U.~Onn, A.~Prasad and L.~Vaserstein in \cite{OnnPrasadVaserstein2006} to contain a wild classification problem in the limit.

The results of \cite{Nevins2011}, \cite{Savin1996} suggest that we should expect all atypical components of depth zero supercuspidal representations to occur in compatible principal series representations, but this is as yet unanswerable with the methods of Section~\ref{S:intertwining}.  To explore this question in an example we use instead explicit realizations of some atypical supercuspidal components of $SL(3,\ratk)$ in Section~\ref{S:example}.

The characters of Deligne-Lusztig cuspidal representations have a uniform description and are well-known; we make use of these in several computations.   It would be useful to extend our results to other families of cuspidal representations, when they exist.  For example, in $\SL(2,\ratk)$, the non-Deligne-Lusztig cuspidal representations give all atypical \emph{irreducible} positive-depth components of all representations \cite{Nevins2013}.  In general we expect that these (smaller) cuspidal representations can be used to define a finer decomposition of the Mackey components of the Deligne-Lusztig supercuspidal representations. 

An eventual goal is the complete decomposition of supercuspidal or principal series representations into irreducible $G_y$-representations.  
As we see in Section~\ref{S:restriction}, this would imply describing the branching rules for the (simple) restriction of cuspidal representations  to a parabolic subgroup.  In Section~\ref{S:DL} we relate this in the Deligne-Lusztig case to questions about the intersection of minisotropic tori with split Levi subgroups. These are interesting open problems in the representation theory of finite groups of Lie type which have as yet been solved only in special cases using CHEVIE~\cite{CHEVIE}, for example.

\noindent {\bf Outline.} In Section~\ref{S:background} we provide a survey of the background required, including several results from Bruhat-Tits theory.  
In Section~\ref{S:stabilizer} we present various properties of pointwise stabilizers of bounded subsets of an apartment, and prove that with few exceptions,  the  Moy-Prasad filtration subgroups are just stabilizer subgroups of certain convex subsets, up to a toral factor.  Section~\ref{S:doublecosets} is devoted to determining a set of double coset representatives $X_{x,y}^+$ indexing the Mackey components of the supercuspidal representations of $G$ for a special vertex $y$ and any vertex $x$, and describing the structure of this set.

In Section~\ref{S:restriction} we prove general results about the restriction of any depth-zero supercuspidal representation of $G$ to a (hyper)special maximal compact subgroup $G_y$.  In Section~\ref{S:DL} we specialize to the case of Deligne-Lusztig representations, proving the coincidence of a large part of their branching rules.  We address principal series representations, proving their extensive intertwining over $G_y$ with Deligne-Lusztig supercuspidal representations, in Section~\ref{S:intertwining}.  We conclude in Section~\ref{S:example} with an example illustrating the use of the many related results in this paper for the group $G=\SL(3,\ratk)$.

\noindent {\bf Acknowledgments.} This research was conducted during a wonderful visiting year at l'Institut de 
Math\'ematiques et Mod\'elisation de Montpellier, Universit\'e de Montpellier II, at the invitation of Ioan Badulescu.  This work also flourished through conversations with Anne-Marie Aubert, Corinne Blondel and C\'edric Bonnaf\'e, and later with Loren Spice, Jeff Adler and Fiona Murnaghan.  Several corrections, improvements and extensions were recommended by the anonymous referee, including particularly the discussion of Gelfand-Kirillov dimension in Section~\ref{S:restriction} and of the relative rarity of the special anisotropic tori of Corollary~\ref{C:allsame}.  It is a pleasure to thank all these people.

\setlength{\parskip}{12pt}

\section{Background: Summary} \label{S:background}

The main references for the background material in this section are \cite{BruhatTits1972,Tits1979}.

\subsection{Notation and conventions} \label{SS:Notation}
Let $\ratk$ be a local non\-arch\-i\-me\-de\-an field of residual characteristic $p\neq 2$.  Its characteristic may be $0$ or $p$.  Its residue field $\resk$ is a finite field of order $q$.
For the sake of brevity we will refer to our field as a $p$-adic field and our group as a $p$-adic group.

Let the integer ring of $\ratk$ be $\R$ and its maximal ideal $\PP$.  Let $\p$ be a uniformizer, and normalize the valuation on $\ratk$ so that $\val(\p)=1$.  
The units of $\R$ admit a filtration by subgroups $U_n$ where $U_0 = \R^\times$ and $U_n = 1+\PP^n$ if  $n>0$.


Given a subgroup $H$ of a group $G$ we denote its center by $Z(H)$ and for any $g\in G$ write $\lconj{g}{H}$ for the group $gHg^{-1}$.  Whenever defined, a representation $(\sigma,V)$ of $H$ is smooth and $V$ is a complex vector space.  We write $V^H$ for the fixed points of $H$ on $V$.  If $g\in G$ then we write $\lconj{g}{\sigma}$ for the corresponding representation of $\lconj{g}{H}$.  Whenever defined, the group $G$ acts on the normalized induced representation $\Ind_H^G \sigma$, or the compactly induced representation $\cind_H^G\sigma$, by right translation.


Define $\Rplus = \real \cup (\real+) \cup \{\infty\} $ as in \cite[6.4.1]{BruhatTits1972}.  For $r\in \real$ we denote by $\lrc{r}$ the least integer $k$ satisfying $k \geq r$ and $\lrc{r+}$ the least integer $k$ with $k>r$.  For $r\in \real$ we also set $\lrf{r} = -\lrc{-r}$.

\subsection{Structure theory} \label{SS:structuretheory}
Let $\GG$ be a connected, simply connected, semisimple algebraic group which is defined and split over $\ratk$.  We write $G = \GG(\ratk)$.  Let $\Stor$ be a maximal torus of $\GG$, split over $\ratk$, and denote the associated root system $\Phi$. Choose positive roots $\Phi^+ \subset \Phi$ and simple roots $\Delta \subseteq \Phi^+$.  Let $\BB$ be the Borel subgroup of $\GG$ defined by $(\Stor,\Phi^+)$ and  $\NN$  the normalizer of $\Stor$ in $\GG$.  We set $S = \Stor(\ratk)$, $B = \BB(\ratk)$ and $N = \NN(\ratk)$.  The corresponding finite Weyl group is $W_0 = N/S$.

Denote by $X_\ast(S)= \Hom_\ratk(\mathbb{G}_m, \Stor)$ the group of $\ratk$-rational cocharacters of $\Stor$, and  $X^\ast(S)= \Hom_\ratk(\Stor,\mathbb{G}_m)$ the group of $\ratk$-rational characters.  Set $S_0 = \{ t \in S \mid \forall \chi \in X^\ast(S), \val(\chi(t))=0 \}$; this is the maximal compact subgroup of $S$.

For each $\alpha\in \Phi \subseteq X^\ast(S)$ we denote by $\alpha^\vee \in \Phi^\vee \subset X_\ast(S)$ the corresponding coroot.   Since $\GG$ is simply connected the lattice $X_\ast(S)$ is spanned by $\Phi^\vee$.

Denote by $\A = \A(\GG, \Stor, \ratk)$ the apartment corresponding to $(\GG,\Stor, \Phi, \ratk)$, which we think of as the affine space under $E=X_\ast(S) \otimes_{\mathbb{Z}} \mathbb{R}$. 
The set of affine roots $\Phi_{af}$ is the set of affine functions $\{\alpha_m = \alpha + m \mid \alpha \in \Phi, m \in \mathbb{Z}\}$ on $\A$; $\alpha$ is the gradient of $\alpha_m$.  The set of hyperplanes $\{\beta=0 \mid \beta \in \Phi_{af}\}$ define the walls of a polysimplicial complex structure on $\A$.
Let $D$ denote the \emph{positive cone} $\{x \in \A \mid \forall \alpha \in \Phi^+, \alpha(x) >0\}$ and let $C$, the \emph{fundamental chamber}, be the unique chamber (also called alcove) in $D$ containing $0\in E$ in its closure.

The affine Weyl group $W$ is generated by the affine reflections $r_\beta$ for $\beta \in \Phi_{af}$, where $r_\beta$ denotes the reflection in the hyperplane $\beta=0$.  We have $W\cong X_\ast(S) \rtimes W_0$; since $\GG$ is simply connected, we also have $W \cong N/S_0$ (called the extended affine Weyl group, in more general settings).  Here, $W_0$ acts as the stabilizer of $0\in E$ and $X_\ast(S)$ acts by translations.  For each $\ell \in X_\ast(S)$ let $t(\ell)\in W$ be its representative in $W$, which we identify with an element of $S \subset N$ when appropriate.
For each $w\in W$ and $\ell \in  X_\ast(S)$ we have $wt(\ell)w^{-1}=t(w\ell)$.

For any $x\in \A$, set $\Phi_x = \{ \beta \in \Phi_{af} \mid \beta(x)=0 \}$
and
$W_x = \langle r_{\beta} \in W \mid \beta \in \Phi_x \rangle$.  Let $\Phi_x^{lin}$ be the set of gradients of elements of $\Phi_x$; since $\GG$ is split over $\ratk$ this is itself a root system.  Choose a base $\Delta_x$ of $\Phi_x^{lin}$ so that the positive roots $\Phi_x^{lin,+}$ coincide with $\Phi_x^{lin} \cap \Phi^+$.  Let $W_x^{lin} \subset W_0$ be the subgroup generated by the linear reflections in elements of $\Phi_x^{lin}$.  
Then the map of $W_x^{lin}$ into $W_x$ given by
\begin{equation} \label{E:wlin}
w \mapsto  t(x-wx)\; w = w \; t(w^{-1}x-x)
\end{equation}
is a group isomorphism.  If $W_x^{lin} = W_0$, then the point $x$ is a vertex and is called a \emph{special vertex} \cite[1.9]{Tits1979}; not all vertices of $\A$ are special in general.  Since $\GG$ is split over $\ratk$, $x$ is special if and only if $\alpha(x)\in \mathbb{Z}$ for all $\alpha \in \Phi$.

\subsection{Filtrations and special subgroups}

Following \cite{MoyPrasad1994}, we associate to each $x \in \A, \alpha \in \Phi$ and $r \in \Rplus$ a subgroup $\GG_\alpha(\ratk)_{x,r}$ of the corresponding root subgroup and, for $r \geq 0$, the subgroup $S_r$ of $S$.   Then the \emph{Moy-Prasad filtration group} at $x$ for $r \in \Rplus_{\geq 0}$ is 
$$
G_{x,r}= \langle S_r,\GG_\alpha(\ratk)_{x,r} \mid \alpha \in \Phi\rangle.
$$
By a choice of pinning we may assume $\GG_\alpha(\ratk)_{x,r} = \GG_\alpha(\PP^{\lrc{r-\alpha(x)}})$ and $S_r = \Stor(U_{\lrc{r}})$.
Note that given $\ell \in X_\ast(S)$, we have
$$
\lconj{t(\ell)}{\GG_\alpha(\ratk)_{x,r}} = \GG_\alpha(\ratk)_{x+\ell,r}= \GG_\alpha(\ratk)_{x,r-\alpha(\ell)}.
$$

Let $\build = \build(\GG,\ratk)$ denote the (reduced) Bruhat-Tits building for $\GG$ over $\ratk$ as in   \cite[7.4.1]{BruhatTits1972}.  
Given any point $y \in \build$, there exist $g\in G$ and $x\in \A$ such that $y=g\cdot x$.  For any $r \in \Rplus_{\geq 0}$, one defines $G_{y,r} := \lconj{g}{G_{x,r}}$; this is independent of choices \cite{MoyPrasad1994}.    Since $\GG$ is semisimple and simply connected, for any $x\in \build$, $G_{x,0}$ coincides with the stabilizer $G_x$ of $x$ in $G$ \cite[\S 3.1]{Tits1979} and is the parahoric subgroup of $G$ associated to $x$.  If $x$ is in an (open) alcove $\Gamma$ then $G_x$ is called an Iwahori subgroup.

In our setting, the 
maximal compact open subgroups of $G$ are exactly the stabilizers of vertices of $\build$.
If $x$ is a special vertex, then  $G_x$ is a good maximal compact subgroup, in the sense that $G$ admits decompositions $G=G_x S G_x$ (Cartan decomposition) and $G = G_xB$ (Iwasawa decomposition). 

Given any $x\in \build$ the group $G_{x,+}:= G_{x,0+}$ is the unipotent radical of the parahoric subgroup $G_x$.
The quotient group $G_{x}/G_{x,+}$ is the group of $\resk$-points of a connected reductive group $\MM_x$ defined over $\resk$ (as in \cite{MoyPrasad1994}).  
Set $\Sres := \Stor(\resk) \subseteq \MM_x(\resk)$.
If $x$ is a hyperspecial vertex (as defined in \cite[1.10]{Tits1979}) then  $\MM_x=\GG$.  Since in our setting $\GG$ is split over $\ratk$, hyperspecial vertices  exist and coincide with the special vertices.  

The maximal compact subgroups which are stabilizers of hyperspecial vertices are distinguished among all maximal compact subgroups in two ways.  First, from their definition it follows that they are isomorphic to $\GG(\R)$. Secondly, they have maximal volume from among all maximal compact open subgroups 
\cite[3.8]{Tits1979}.   In this paper we choose to restrict to a maximal compact subgroup which is the stabilizer of a (hyper)special vertex, always denoted $y$.

To reduce notational burden, we write $\G_x = G_{x}/G_{x,+}$ for $\MM_x(\resk)$ and refer to parabolic subgroups ($\Para$ and $\B$) and tori ($\T$) of $\G_x$ without reference to the algebraic group $\MM_x$.  This is unfortunate in one case arising in Section~\ref{S:DL}; let us define the needed terms here.  Let $s \in \G_x$ be semisimple and let $\mathbb{C}_s$ denote its centralizer, which is a reductive subgroup of $\MM_x$, and $\mathbb{C}_s^\circ$ its connected component subgroup.  This coincides with $\mathbb{C}_s$ if $\MM_x$ is simply connected \cite{Carter1985}. We define $C_{\G_x}^\circ(s)= \mathbb{C}_s^\circ(\resk)$.  Note that if $s \in Z(\G_x)$ then $\mathbb{C}_s^\circ=\MM_x$ and so $C_{\G_x}^\circ(s)=\G_x$.

\subsection{Representations of $G$}

Given an irreducible admissible representation $\pi$ of $G$ on a complex vector space $V$, the \emph{depth} of $\pi$ is a rational number defined as the least $r\in \real_{ \geq 0}$ such that there exists $x\in \build(\GG,\ratk)$ for which $V$ contains vectors invariant under $G_{x,r+}$ \cite{MoyPrasad1994}. 
 Where appropriate, we also refer to the depth of a representation of $G_x$, for fixed $x$.  If $x$ is a special vertex then the depth of any representation of $G_x$ is a nonnegative integer.

By Jacquet's theorem, every irreducible admissible representation of $G$ occurs as a subrepresentation of $\Ind_P^G \sigma$, for some parabolic subgroup $P$ with Levi decomposition $MN$ and supercuspidal representation $\sigma$ of $M$ (extended trivially across $N$).  In case $P=B$, a Borel subgroup, the representation $\sigma$ is simply a character $\chi$ of a split torus $S$ and the representation $\Ind_B^G \chi$, which may fail to be irreducible, is called a principal series representation.

The classification of (irreducible) supercuspidal representations is not yet complete.  It is a lasting conjecture, proven now in many cases, that all supercuspidal representations of depth $r$ are compactly induced from a compact open subgroup. 
 In case $r=0$ this has been proven; more precisely L.~Morris \cite{Morris1999} and A.~Moy and G.~Prasad \cite{MoyPrasad1996} proved that all depth-zero supercuspidal representations of $G$ are given by
\begin{equation} \label{E:supercuspidal}
\pi= \cind_{G_x}^G \tau
\end{equation}
for some vertex $x\in \build$ and inflation $\tau$ of a cuspidal representation of $\G_x$.  Among these cuspidal representations $\tau$ are the Deligne-Lusztig cuspidal representations, whose characters are well-known; see Section~\ref{S:DL}.

\section{Stabilizers of subsets of $\A$} \label{S:stabilizer}

Let $\Omega$ be a bounded subset of $\build$.  Its convex closure
 $\overline{\Omega}$ is the union of the closures of all the facets of $\build$ meeting $\Omega$.  
The pointwise stabilizer of $\Omega$ is $G_\Omega= \cap_{x \in \Omega} G_x$ and it coincides with $G_{\overline{\Omega}}$ \cite[Prop 2.4.13]{BruhatTits1972}.
Given two points $x,y\in \build$, we have $G_x \cap G_y = G_{[x,y]}$, where $[x,y]$ is the unique geodesic joining $x$ and $y$, which is a line in any apartment containing both points \cite[Prop 2.5.4]{BruhatTits1972}. 
 From these facts one concludes that if $F$ is a facet such that $[x,y] \cap F \neq \emptyset$, then $G_{[x,y]} \subseteq G_F$.

F.~Bruhat and J.~Tits give the following description of $G_\Omega$ if $\Omega \subseteq \A$ \cite[\S 6.4]{BruhatTits1972}.  

\begin{proposition}\label{P:stabilizer}
Suppose $\Omega$ is a bounded subset of $\A$.  For each $\alpha \in \Phi$, define 
$$
f_\Omega(\alpha) = \max\{ \lrc{-\alpha(x)} \mid x \in \Omega\}.
$$  
Then $G_\Omega = S_0U_\Omega$ where $U_\Omega = \langle \GG_\alpha(\PP^{f_\Omega(\alpha)}) \mid \alpha \in \Phi \rangle$.  Furthermore, if $\Omega$ contains an open set of $\A$ then for any order on $\Phi$ the product map
\begin{equation} \label{E:stabilizer}
S_0 \times \prod_{\alpha \in \Phi}\GG_\alpha(\PP^{f_\Omega(\alpha)}) \to G_\Omega
\end{equation}
is a bijection.
\end{proposition}

More generally, suppose $\Omega$ is contained in an intersection of affine root hyperplanes.  The corresponding gradients then form a subrootsystem $\Phi_\Omega$. If we define $L_\Omega = \langle S_0, \GG_\alpha(\PP^{f_\Omega(\alpha)}) \mid \alpha \in \Phi_\Omega \rangle$ then $G_\Omega$ is in bijection with  $L_\Omega \times \prod_{\alpha \in \Phi\setminus \Phi_\Omega}\GG_\alpha(\PP^{f_\Omega(\alpha)})$.

As a particular consequence we note the following.   For $\Omega \subset \A$, write $int(\Omega)$ for the topological interior of $\Omega$.

\begin{corollary} \label{C:interior}
Let $\Omega \subset \A$ be a bounded set such that $x\in int(\Omega)$.  Then 
in the factorization $G_\Omega = S_0 U_\Omega$ we have $U_\Omega \subseteq G_{x,+}$.
\end{corollary}

\begin{proof}
Since $G_{x,+}$ is generated by $S_1$ and the groups $\GG_\alpha(\PP^{\lrc{-\alpha(x)+}})$, by Proposition~\ref{P:stabilizer} it suffices to show that for all $\alpha \in \Phi$, $f_{\Omega}(\alpha)>-\alpha(x)$.  
Since $x\in int(\Omega)$ there exists some $z\in \Omega$ such that $\alpha(z)<\alpha(x)$, whence $f_\Omega(\alpha) \geq \lrc{-\alpha(z)}> -\alpha(x)$.
\end{proof}

We next wish to describe the relationship between  subgroups $G_{\Omega}$, with $\Omega \subseteq \A$, and  Moy-Prasad filtration subgroups $G_{x,r}$.  We begin by setting some notation.

If an irreducible root system $\widetilde{\Phi}$ has two root lengths let $\widetilde{\Phi}^l$ be the set of its long roots and $\widetilde{\Phi}^s=\widetilde{\Phi}\setminus \widetilde{\Phi}^l$; otherwise, let $\widetilde{\Phi}^s=\widetilde{\Phi}^l=\widetilde{\Phi}$.   Given a root system $\Phi$ with irreducible components $\widetilde{\Phi}_i$, for $1\leq i \leq m$, define $\Phi^l = \cup_i \widetilde{\Phi}_i^l$ and $\Phi^s = \cup_i \widetilde{\Phi}_i^s$.  Note that $\Phi$, $\Phi^l$ and $\Phi^s$ all have the same rank.

Given $x\in \A$ and $r \in \real_{\geq 0}$, define
\begin{equation} \label{E:omegaar}
\Omega_{x}(\A,r) = \{ z \in \A \mid \forall \alpha \in \Phi, \vert \alpha(x)-\alpha(z) \vert \leq r \}.
\end{equation}
Define $\Omega^l_{x}(\A,r)$ and $\Omega^s_x(\A,r)$ by replacing $\Phi$ in \eqref{E:omegaar} with $\Phi^l$ and $\Phi^s$, respectively. 

\begin{proposition} \label{P:gxromega}
Let $x \in \A$ and $r  \in \real_{\geq 0}$.  Then
\begin{equation}\label{E:shortlong}
G_{\Omega^s_x(\A,r)} \subseteq S_0G_{x,r} \subseteq G_{\Omega^l_x(\A,r)}=G_{\Omega_x(\A,r)}.
\end{equation}
Moreover, whenever the root system $\Phi$ does not contain an irreducible component of type $G_2$ the second inclusion is an equality, that is, $S_0G_{x,r} = G_{\Omega_x(\A,r)}$.
\end{proposition}

\begin{proof}
First note that $\Omega^l_x(\A,r)=\Omega_x(\A,r)$.  Namely, given $z \in \Omega^l_x(\A,r)$, choose a positive system $\Phi^{(+)}$ for which $z-x$ is in the closure of the positive cone and let $\theta^{(+)} \in \Phi^l$ be the corresponding highest (long) root.  Then for each $\beta \in \Phi$, $\vert \beta(x-z) \vert \leq \theta^{(+)}(z-x)\leq r$, so $z \in \Omega_x(\A,r)$.  Clearly also $\Omega^l_x(\A,r)\supseteq \Omega_x(\A,r)$.  Hence $G_{\Omega^l_x(\A,r)}=G_{\Omega_x(\A,r)}$.

If $r=0$ the groups appearing in \eqref{E:shortlong} are all equal and there is nothing to show, so suppose $r>0$.  Each group is generated by $S_0$ and certain subgroups of the root groups; thus it suffices to show the inclusions on each root subgroup.

Let $z\in \Omega_x(\A,r)$.  Then for each $\alpha \in \Phi$ we have $-\alpha(z)\leq r-\alpha(x)$, whence $\GG_\alpha(\PP^{\lrc{r-\alpha(x)}}) \subseteq \GG_\alpha(\PP^{\lrc{-\alpha(z)}})$.  It follows that $G_{x,r} \subseteq \cap_{z\in \Omega_x(\A,r)}G_z = G_{\Omega_x(\A,r)}$, and the second inclusion holds.

Now consider the first inclusion.  It suffices to show that for all $\alpha \in \Phi$ there exists $z_\alpha \in \Omega^s_x(\A,r)$ such that $-\alpha(z_\alpha)\geq r-\alpha(x)$ (as then $f_{\Omega^s_x(\A,r)}(\alpha) \geq r-\alpha(x)$).

First suppose $\alpha \in \Phi^s$.  Then $\alpha$ lies in a unique irreducible component ${\Phi^s}'$ of $\Phi^s$, corresponding to a subspace $E'$ of $E=X_\ast(S)\otimes_{\mathbb{Z}}\real$.  Let $\Delta'$ be the base of  ${\Phi^s}'$ with respect to which $\alpha$ is the highest root and let $H_{\alpha,r}'$ denote the (nonempty) intersection of the hyperplane $\alpha = r$ with the positive cone $D'$ defined by $\Delta'$.  Choose $v \in H_{\alpha,r}' \subset E'$.  Then for all $\beta \in  {\Phi^s}'$ we have $\vert \beta(v) \vert \leq \alpha(v) = r$, and for all $\beta \in \Phi^s \setminus {\Phi^s}'$ the orthogonality of irreducible components ensures $\beta(v)=0 \leq r$.  Therefore the element $z_\alpha = x-v$ satisfies our requirements.

Now let $\alpha \in \Phi \setminus \Phi^s$.  Let $\Phi''$ denote the irreducible component of $\Phi$ containing $\alpha$, $\Delta''$ the base with respect to which $\alpha$ is the highest root, and $D''$ the corresponding positive cone.  Let $\alpha_0 \in \Phi^s \cap \Phi''$ be the corresponding highest short root and define $H_{\alpha_0,r}'$ as in the preceding paragraph.  For any $v$ in the nonempty intersection $H_{\alpha_0,r}'\cap D''$ we have $\alpha(v) \geq \alpha_0(v)=r$ and, as argued above, for all $\gamma \in \Phi^s$, $\vert \gamma(v)\vert \leq r$.  Therefore $z_\alpha = x-v \in \Omega^s_x(\A,r)$ and satisfies  $-\alpha(z_\alpha)\geq r-\alpha(x)$, as required.

Now consider the final assertion.  If $\Phi$ is simply-laced then equality holds because $\Omega^s_x(\A,r)=\Omega_x(\A,r)$.   Otherwise by the preceding arguments it suffices to show that in each non-simply-laced irreducible root system except $G_2$, there exists a short root $\alpha$ and a vector $v$ such that $\alpha(v)=r$ and for all $\beta \in \Phi$, $\vert \beta(v)\vert \leq r$.  This is easily verified case-by-case.  
\end{proof}

We remark that equality fails on the simple system of type $G_2$ because the boundary of $\Omega_x^s(\A,r)$ does not meet the boundary of  $\Omega_x(\A,r)$.

\section{The double coset space $G_{y}\backslash G /G_{x}$} \label{S:doublecosets}

We begin by recalling a result about generalized BN-pairs \cite[Proposition 7.4.15]{BruhatTits1972}. 

\begin{proposition}\label{P:doublecosets}
For $i=1,2$ let $\Omega_i$ denote a nonempty subset of $\A$, $G_i$ its pointwise stabilizer in $G$, $N_i$ the pointwise stabilizer of $\Omega_i$ in $N$, and $\widehat{W}_i$ its image in $W=N/S_0$.  Then the natural map
$$
\widehat{W}_1\backslash W / \widehat{W}_2 \to G_1\backslash G / G_2
$$
is bijective.
\end{proposition}

\begin{corollary}
Let $x,y$ be vertices of $\A$.  Then
$$
G_{y}\backslash G /G_{x} \cong W_y \backslash W / W_x.
$$
\end{corollary}

\begin{proof}
By Proposition~\ref{P:doublecosets}, it suffices to note that for any vertex $z \in \A$, the group $\widehat{W}_z=(N \cap G_z)/S_0$ coincides with $W_z$, the group generated by the reflections in the affine hyperplanes through $z$.  This follows in our case from \cite[7.1.3]{BruhatTits1972}. 
\end{proof}

Let $D_x:= \{z \in \A \mid \forall \alpha \in \Phi_x^{lin,+}, \alpha(z)>0\}$ denote the positive cone for $\Phi_x^{lin,+}$, whose closure is a fundamental domain for the action of $W_x^{lin}$ on $\A$.  Let $\Upsilon_x = \{ w \in W_0 \mid w D \subseteq D_x\}$; then $\overline{D_x} = \cup_{w\in \Upsilon_x}w\overline{D}$.

\begin{proposition} \label{P:doublecosetreps}
Suppose $y$ is special.  A set of double coset representatives for 
$W_y \backslash W / W_x$
is given by 
\begin{align*}
X_{x,y}^+ &= X_\ast(S) \cap (y-x+\overline{D_x})\\
 &= \{\ell \in X_\ast(S) \mid \forall \alpha\in \Phi_x^{lin,+}, \alpha(\ell) \geq \alpha(y-x) \}. 
\end{align*}
\end{proposition}

\begin{proof}
Since $W_y \cong W_0$, we have $W_y\backslash W \cong X_\ast(S)$, so a set of double coset representatives may be chosen from those elements of $X_\ast(S) \subset \A$ lying in a suitable translate of the fundamental domain $\overline{D_x}$ for the action of $W_x^{lin}$.  To identify the correct translate, note that for every $w_y\in W_y$ and $w_x \in W_x$ with corresponding $w_1\in W_0$ and $w_2\in W_x^{lin}$ (as in \eqref{E:wlin}), and any $\ell \in X_\ast(S)$, we have
$$
w_yt(\ell)w_x = w_1 w_2 t(w_2^{-1}(\ell+x-y)+w_2^{-1}w_1^{-1}y-x).
$$
This is a translation exactly when $w_1w_2=1$, whence by varying $w_2$ one has within this orbit a unique translation $t(\ell)$ corresponding to $\ell+x-y \in \overline{D_x}$.
\end{proof}

For example, if $y=x$ is special then $X_{y,y}^+ = X_+$, the set of semidominant cocharacters.  

\begin{remark}
If $x\neq y$, then there is some $\alpha \in \Phi_x^{lin}$ for which $\alpha(x-y)\neq 0$, so that $X_{x,y}^+ \neq X_+$.  More generally $X_{x,y}^+ = X_+ + (y-x)$ if and only if $x-y \in X_\ast(S)$, which will not arise if $x,y$ are chosen in distinct orbits under $G$, for example.
\end{remark}

\begin{definition}
Let $int(X_{x,y}^+) = X_\ast(S) \cap (y-x+D_x)$ and $\partial(X_{x,y}^+)= X_{x,y}^+\setminus int(X_{x,y}^+)$,
which we call the interior and the boundary of $X_{x,y}^+$, respectively.  
\end{definition}

We record some key properties of the interior of $X_{x,y}^+$ in two lemmas.

\begin{lemma} \label{L:upsilon}
The boundary of $y-x+\overline{D}$ does not meet $int(X_{x,y}^+)$.  More generally, 
$$
int(X_{x,y}^+) = \bigsqcup_{w \in \Upsilon_x} X_{x,y}^+ \cap (y-x+wD).
$$
\end{lemma}

\begin{proof}
Since $\Phi_x^{lin}\subseteq \Phi$, $\overline{D_x}=\cup_{w \in\Upsilon_x} w\overline{D}$ and thus $X_{x,y}^+ \subset \cup_{w \in\Upsilon_x}\left(y-x+w\overline{D}\right)$.
Fix $w\in \Upsilon_x$ and suppose $\ell \in X_{x,y}^+ \cap (y-x+w(\overline{D}\setminus D))$.  Then $x-y+\ell \in w(\overline{D}\setminus D)$ so there exists $\alpha \in \Phi$ such that $\alpha(x-y+\ell)=0$.  But as $y$ is special and $\ell \in X_\ast(S)$, this implies $\alpha(x) \in \mathbb{Z}$, whence $\alpha \in \Phi_x^{lin}$.  Consequently, $x-y+\ell \in \partial(X_{x,y}^+)$.
\end{proof}

Finally, recall that for a vertex $u$ and point $v\neq u$ in $\A$, the geodesic $[u,v]$ meets a unique facet $F_u \neq \{u\}$ of $\A$ whose closure contains $u$.

\begin{lemma} \label{L:alcove}
If $\ell \in int(X_{x,y}^+)$ then the convex closure of $[y,x+\ell]$ in $\A$ contains unique alcoves adjacent to each endpoint.
\end{lemma}

\begin{proof}
This is the observation that, for interior $\ell$, $[y,x+\ell]$ is contained in no wall of $\A$, and hence that the facets $F_y$ and $F_{x+\ell}$ at each endpoint are alcoves.  Namely, suppose $\ell \in X_{x,y}^+$ is such that for some $\alpha \in \Phi$, and some $m\in \mathbb{Z}$, $\alpha(y)=\alpha(x+\ell) = m$.  Then $\alpha(x)=m-\alpha(\ell) \in \mathbb{Z}$, so in fact $\alpha\in \Phi_x^{lin}$.  Since $\alpha(\ell)=\alpha(y-x)$, we deduce $\ell \in \partial(X_{x,y}^+)$.  
\end{proof}

\section{Restrictions of Supercuspidal Representations to $G_y$} \label{S:restriction}

For reference we cite a consequence of Mackey theory  for compactly induced representations derived from \cite{Kutzko1977}.

\begin{lemma} 
Let $G$ be the $\ratk$-points of a linear algebraic group defined over $\ratk$, with a compact open subgroup $K$ and a compact-mod-center subgroup $H$.  Let $\rho$ be
a smooth representation of $H$ such that 
$\pi = \cind_H^G \rho$ is admissible.   For any $t \in K\backslash G / H$, the subspace of $\cind_H^G\rho$ consisting of vectors supported on the double coset $Ht^{-1}K$ is $K$-invariant, and as a representation of $K$ is isomorphic to 
$\Ind_{K \cap \lconj{t}{H}}^K \lconj{t}{\sigma}$.
Thus we have
\begin{equation} \label{E:decomp}
\Res_K \cind_H^G \sigma \cong \bigoplus_{t \in K\backslash G / H} \Ind_{K \cap \lconj{t}{H}}^K \lconj{t}{\sigma}.
\end{equation}
\end{lemma}

In our case, let $H = G_x$ and $K=G_y$, for vertices $x,y \in \A$ with $y$ special.
Given an irreducible supercuspidal representation $\pi = \cind_{G_x}^G \tau$ 
we therefore have
\begin{align*}
\Res_{G_y}\pi &= \Res_{G_y}\cind_{G_x}^G \tau \\
&\cong
\bigoplus_{t \in G_{y}\backslash G /G_{x} } \Ind_{G_{y} \cap \lconj{t}{G_{x}}}^{G_{y}} \lconj{t}{\tau}.
\end{align*}
By Proposition~\ref{P:doublecosetreps}, we may choose the representatives of 
$G_y\backslash G / G_x$ to be $\{t(\ell) \mid  \ell \in X_{x,y}^+\}$, whence $G_{y} \cap \lconj{t(\ell)}{G_{x}} = G_y \cap G_{x+\ell} = G_{[y,x+\ell]}$. 
Thus we may rewrite the sum above as 
\begin{equation} \label{E:decomp2}
\Res_{G_y}\pi  \cong \bigoplus_{\ell \in X_{x,y}^+} \Ind_{G_{[y,x+\ell]}}^{G_{y}}\lconj{t(\ell)}{\tau}.
\end{equation}
We refer to the representation $\pi_\ell = \Ind_{G_{[y,x+\ell]}}^{G_{y}}\lconj{t(\ell)}{\tau}$ as a \emph{Mackey component} of $\Res_{G_y}\pi$.  Note that this is not an irreducible representation in general.

Suppose from now on that $\tau$ has depth zero, and let us record some basic properties of the Mackey components $\pi_\ell$.

\begin{proposition} \label{P:degree}
Suppose $\ell \in X_{x,y}^+$ and set $\pi_\ell = \Ind_{G_{[y,x+\ell]}}^{G_y}\lconj{t(\ell)}{\tau}$.  For $v\in \A$ define
$$
\eta(v) = \sum_{\alpha \in \Phi : \alpha(v)>0} \lrc{\alpha(v)-1}.
$$
Then 
$$
\deg(\pi_\ell) = \deg(\tau) q^{\eta(x-y+\ell)} \left\vert \G_y/\Para_{x+\ell} \right\vert,
$$
where $\Para_{x+\ell}$ is the parabolic subgroup of $\G_y$ corresponding to the facet containing $[y,x+\ell]$.  If $\ell \in int(X_{x,y}^+)$ then $\Para_{x+\ell}$ is a Borel subgroup; if in addition $x$ is special then $\eta(x-y+\ell)=2\rho(x-y+\ell)-\vert \Phi^+ \vert$, where $\rho = \frac12 \sum_{\alpha \in \Phi^+}\alpha$.
\end{proposition}

\begin{proof}
We suppose $\ell \in X_{x,y}^+$ and compute $[G_y \colon G_{[y,x+\ell]}]$.  Let $F$ denote the facet of $\A$ containing $y$ in its closure meeting $[y,x+\ell]$; then $G_{[y,x+\ell]}\subseteq G_F \subseteq G_y$ and the image of $G_F$ in $\G_y$ is a parabolic subgroup which we denote $\Para_{x+\ell}$.  We have $[G_y:G_F]=[\G_y:\PP_{x+\ell}]$.   By Lemma~\ref{L:alcove} if $\ell \in int(X_{x,y}^+)$ then $F$ is an alcove and $\PP_{x+\ell}$ is a Borel subgroup; since these are all conjugate $[\G_y:\PP_{x+\ell}]$ is independent of $x+\ell$ in this case. 

The remaining factor $[G_{F}\colon G_{[y,x+\ell]}]$ can be computed directly (Proposition~\ref{P:stabilizer}).  Set $\Omega = [y,x+\ell]$; then for each $\alpha \in \Phi$ we have $f_\Omega(\alpha) = \max\{-\alpha(y), \lrc{-\alpha(x+\ell)}\}$.
If $\alpha(\ell+x-y)\geq 0$ then $f_\Omega(\alpha) = -\alpha(y) = f_F(\alpha)$; otherwise, $f_\Omega(\alpha)=\lrc{-\alpha(x+\ell)}$ whereas $f_F(\alpha)=-\alpha(y)+1$.  Therefore what remains in the quotient is 
\begin{equation} \label{E:index}
\vert G_{F}/G_{[y,x+\ell]} \vert = \prod_{\alpha \in \Phi:\alpha(\ell+x-y)>0} q^{\lrc{\alpha(x-y+\ell)-1}} = q^{\eta(x-y+\ell)}.
\end{equation}
If $x$ is special, then $\alpha(x-y+\ell)\in \mathbb{Z}$ for all roots $\alpha$.
\end{proof}

\begin{theorem} \label{T:depth}
Let $\ell \in X_{x,y}^+$.  Set 
$$r_0= \max\{ \beta(x-y+\ell) \mid \beta \in \Delta_x \}$$
 and
$$s_0 =  \max\{ \lrf{\alpha(x-y+\ell)} \mid \alpha \in \Phi \}.$$
Then the depth $d$ of an irreducible subrepresentation of  $\Ind_{G_{[y,x+\ell]}}^{G_y}\lconj{t(\ell)}{\tau}$ satisfies $r_0 \leq d \leq s_0$.
\end{theorem}

\begin{proof}
Let $\ell \in X_{x,y}^+$ and set $\pi_\ell = \Ind_{G_{[y,x+\ell]}}^{G_y}\lconj{t(\ell)}{\tau}$.  If the space of $\tau$ is denoted $V_\tau$ then the space of $\pi_\ell$ is $V_\ell = \{f\colon G_y \to V_\tau \mid \forall h\in G_{[y,x+\ell]},\forall g\in G_y, f(hg)= \lconj{t(\ell)}{\tau}(h)f(g)\}$.  It suffices to prove that $V_\ell^{G_{y,r_0}} = \{0\}$ and $V_\ell^{G_{y,s_0+}} = V_\ell$. 

By construction, $\tau$ is trivial on $G_{x,+}$, and thus $\lconj{t(\ell)}{\tau}$ is trivial on $\lconj{t(\ell)}{G_{x,+}} = G_{x+\ell,+}$.  Given a nonnegative integer $s$, the subgroup $G_{y,s+}$ is contained in $G_{x+\ell,+}$ if and only if for each $\alpha \in \Phi$, we have $\lrc{(s-\alpha(y))+} \geq \lrc{-\alpha(x+\ell)+}$.  As $\alpha(y), \alpha(\ell)\in \mathbb{Z}$ this condition is equivalent to $s \geq \lrf{\alpha(y-x-\ell)}$.  Set $s_0=\max\{ \lrf{\alpha(x-y+\ell)} \mid \alpha \in \Phi \}$; this is nonnegative since $\alpha(x-y+\ell)\geq 0$ for $\alpha \in \Phi_x^{lin,+}$.   Thus  $G_{y,s_0+}$ is a normal subgroup of $G_y$ contained in the kernel $G_{x+\ell,+}$ of $\lconj{t(\ell)}{\tau}$, whence $V_\ell^{G_{y,s_0+}} = V_\ell$. 

Now let $\HH$ be the unipotent radical of a proper parabolic subgroup $\Para$ of $\G_x$.   Since $\tau$ is a cuspidal representation of the finite group $\G_x$, $V_\tau^{\HH} = \{0\}$.  Let $H \subseteq G_x$ be a subgroup satisfying $H/(H \cap G_{x,+}) = \HH$.  
Using elementary arguments, and the normality of $G_{y,r}$ in $G_y$,  one can show that if $\lconj{t(\ell)}{H} \subseteq G_{y,r}$ then $V_\ell^{G_{y,r}} = \{0\}$.

We choose our various suitable $\HH$ and $H$ as follows.  Each proper subset $\Delta'$ of $\Delta_x$ defines two proper parabolic subgroups of $\G_x$: the standard parabolic $\Para_{\Delta'}$ and its opposite $\Para^{op}_{\Delta'}$.  Let $\HH$ be the unipotent radical of $\Para^{op}_{\Delta'}$.  If $\Phi'$ is the subrootsystem of $\Phi_x^{lin}$ generated by $\Delta'$, then $\HH$ is spanned by the root subgroups of $\G_x$ corresponding to $\{-\alpha \mid \alpha \in \widetilde{\Phi}=\Phi_x^{lin,+}\setminus \Phi'\}$.
We may choose $H = \langle \GG_{-\alpha}(\ratk)_{x,0} \mid \alpha \in \widetilde{\Phi} \rangle \subseteq G_x$ as our lift of $\HH$.  Note that if $x$ is not special then $H$ is not necessarily contained in the unipotent radical of a parabolic subgroup of $G$.  

We have $\lconj{t(\ell)}{H} = \langle \GG_{-\alpha}(\ratk)_{x,\alpha(\ell)} \mid \alpha \in \widetilde{\Phi} \rangle$.  Thus $\lconj{t(\ell)}{H}\subseteq G_{y,r}$ if and only if for each $\alpha \in \widetilde{\Phi}$, $\lrc{r+\alpha(y)} \leq \lrc{\alpha(x+\ell)}$.  Since $\alpha$ takes integral values on $x,y$ and $\ell$, this simplifies to $r \leq \alpha(x-y+\ell)$.  Each simple root $\beta \in \Delta_x$ takes nonnegative values on $x-y+\ell$; by construction of $\widetilde{\Phi}$ we deduce that $\min\{\alpha(x-y+\ell)\mid \alpha \in \widetilde{\Phi}\}$ is attained on some simple root $\beta \in \Delta_x \setminus \Delta' \subseteq \widetilde{\Phi}$, whence $V_\ell^{G_{y,\beta(x-y+\ell)}} = \{0\}$.
Conversely, given $\beta \in \Delta_x$, choosing $\Delta' = \Delta_x \setminus \{\beta\}$ ensures that  $V_\ell^{G_{y,\beta(x-y+\ell)}} = \{0\}$.  We conclude that $r_0 = \max\{ \beta(x-y+\ell) \mid \beta \in \Delta_x\}$ has the property required.
\end{proof}

From the Harish-Chandra--Howe character formula one deduces 
\cite{BarbaschMoy1997, Savin1996} that for integral $m$, $\dim V_\pi^{G_{y,m}}$ grows as a polynomial in $q$ of degree $md_\pi+c$ where $d_\pi = \frac12 \dim(\mathcal{O})$ for a maximal orbit $\mathcal{O}$ in the wave front set of $\pi$ and $c$ is a constant depending on the orbit and the Harish-Chandra coefficients.  This integer $d_\pi$ is called the \emph{Gelfand-Kirillov (GK) dimension} of $\pi$.  One consequence of Theorem~\ref{T:depth} is a pair of bounds on the GK dimension of $\pi$.  We begin with a lemma.

In the following, given a representation $\varphi$ whose degree is expressed as a polynomial in $q$ of polynomial degree $n$, we set $d_q(\varphi)=n$.

\begin{lemma} \label{L:asymptotic}
Let $\tau, \pi$ be as above. Then for any $\ell \in X_{x,y}^+$ there is a unique $w \in W_0$ such that 
$$
d_q(\pi_\ell) = 2\rho(w(x-y+\ell)) + d_q(\tau) + \varepsilon(\ell)
$$
where $0 \leq \varepsilon(\ell) < \vert \Phi^+ \vert$, and $\varepsilon(\ell)=0$ if $x$ is special.
\end{lemma}

\begin{proof}
Recall the notation of Proposition~\ref{P:degree}.  There is a unique $w^{-1}\in \Upsilon_x \subset W_0$ such that $v := w(\ell+x-y) \in \overline{D}$.  Then $\eta(v) = \sum_{\alpha\in \Phi^+}\lrc{\alpha(v)} - n$ where $n = \vert \{ \alpha \in \Phi^+ \mid \alpha(v) \neq 0\} \vert$.  Taking into account the case that some of the $n$ terms $\alpha(v)$ are nonintegral it follows that $2\rho(v) -n \leq \eta(v)< 2\rho(v)$.  As $\vert \G_y/\Para_{x+\ell} \vert$ is a polynomial in $q$ of this same degree $n$, and $n \leq \vert \Phi^+ \vert$, the lemma follows.
\end{proof}

Write $r_0(\ell)$ and $s_0(\ell)$ for the lower and upper bounds on the depth of an irreducible component of $\pi_\ell$, respectively, as in Theorem~\ref{T:depth}.  It follows from the theorem that
$$
\bigoplus_{s_0(\ell)<m} \pi_\ell \subseteq V_\pi^{G_{y,m}} \subseteq \bigoplus_{r_0(\ell)< m} \pi_\ell.
$$
Since the number of terms in each sum is independent of $q$, it follows that for $q$ sufficiently large, $d_\pi(m) := d_q(V_\pi^{G_{y,m}})$ satisfies
$$
\max\{d_q(\pi_\ell) : s_0(\ell)<m \} \leq d_\pi(m) \leq \max\{d_q(\pi_\ell) : r_0(\ell)<m \}.
$$
From Lemma~\ref{L:asymptotic} it follows that for the purpose of estimating the GK-dimension $d_\pi$ of $\pi$, we may replace each $d_q(\pi_\ell)$ with $2\rho(w(\ell+x-y))$.

Thus it suffices to compute these maximi over each irreducible component of $\Phi$.  We assume for the remainder that $\Phi$ is irreducible and, to obtain strong explicit results, that $x$ is special.  Then $D_x=D$, $w=1$ and $\Delta_x=\Delta = \{ \alpha_1, \cdots, \alpha_n\}$.

Let $\alpha_0$ denote the highest root of $\Phi$ and write $\alpha_0=\sum_i c_i\alpha_i$; so  $s_0(\ell) = \alpha_0(x-y+\ell)$.  Write $2\rho = \sum_i \gamma_i \alpha_i$.  First note that by linearity, the extremum
$$
L_m:= \max\{2\rho(v) \mid v\in \overline{D}, \alpha_0(v)=m-1\}
$$
is attained at some vertex $v$ of $\overline{D} \cap \{v \mid \alpha_0(v)= m-1\}$.  These vertices correspond to $v_i$ such that $\alpha_j(v_i)=0$ unless $i=j$ and thus $\alpha_i(v_i)=(m-1)/c_i$.  Evaluating $2\rho$ on each of these vertices yields  
$L_m = (m-1) \max\{\gamma_i/c_i\mid 1\leq i \leq n\}$.   Fix $i$ at which this maximum is attained.  Since there is a constant radius $r$ (related to the size of a Voronoi cell of $X_\ast(S)$) such that for each $m$ sufficiently large, there is an $\ell \in X_{x,y}^+$ with $\vert x-y+\ell - v_i \vert < r$, we deduce by the linearity of $\rho$ that $L_m$ differs from $\max\{2\rho(x-y+\ell) \mid s_0(\ell)<m\}$ by a bounded factor which does not grow with $m$.  

Similarly, the maximum 
$$
U_m := \max\{2\rho(v)\mid \forall 1\leq i\leq n, 0 \leq \alpha_i(v)\leq m-1\}
$$
is attained when $\alpha_i(v)=m-1$ for all $i$, yielding $U_m = (m-1)\sum_i \gamma_i$.  Since $r_0(\ell)=\max\{\alpha_i(x-y+\ell)\}$,  we again deduce that $U_m$ differs from $\max\{2\rho(x-y+\ell) \mid r_0(\ell)<m\}$ by a factor which does not grow with $m$.

We have thus proven the following result.

\begin{corollary} \label{C:GK}
Suppose $\Phi$ is irreducible and $x$ is special.  Let $\alpha_0=\sum_i c_i\alpha_i$ and $2\rho=\sum_i\gamma_i\alpha_i$.  Then the GK-dimension $d_\pi$ of $\pi$ is bounded by
$$ 
\lfloor \max\{\gamma_i/c_i \mid 1 \leq i \leq n\}\rfloor \leq d_\pi \leq  \sum_i \gamma_i.
$$
\end{corollary}

When $x$ is non-special, a similar argument gives the same lower bound but, in general, a tighter upper bound.  

\begin{remark}
Note that, generally speaking, the upper bound in Corollary~\ref{C:GK} is cubic in the rank $n$ of $\Phi$ whereras the lower bound is quadratic in $n$. The GK-dimension, being half the dimension of a nilpotent adjoint orbit of $G$, is bounded above by $\vert \Phi^+ \vert$,  which is quadratic in $n$.  Consequently, Corollary~\ref{C:GK} 
 will not generally in and of itself suffice to identify the GK-dimension.
\end{remark}

\begin{example}
For $G=\SL(2,\ratk)$, with $y=0$, one always has $r_0=s_0$.  Indeed, the depths of the irreducible components of $\pi_\ell$ were shown to be exactly $\delta(\ell) = \alpha(x-y+\ell) = x+\alpha(\ell)$ in \cite[\S 5]{Nevins2013}; and evidently $d_\pi = 1$, corresponding to half the dimension of a principal nilpotent orbit.
\end{example}

\begin{example}
For $G=\mathrm{Sp}(4,\ratk)$, with $\Phi^+ = \{ \alpha, \beta, \alpha+\beta, 2\alpha+\beta\}$, if $x$ is the non-special vertex of $C$ then $\Delta_x = \{\beta, 2\alpha+\beta\}$.  Since the highest root of $\Phi$ is a simple root of $\Delta_x$, $r_0=s_0$ and  the depth of each irreducible subrepresentation of $\pi_\ell$ is exactly $\max\{ \beta(x-y+\ell), (2\alpha+\beta)(x-y+\ell) \} \in \mathbb{Z}$.  Via the argument of the proof of Corollary~\ref{C:GK} we deduce that $d_\pi = 3$.

If $x$ is special, however, then the lower and upper bounds given in Theorem~\ref{T:depth} cannot coincide, and Corollary~\ref{C:GK} yields only $3 \leq d_\pi \leq 7$.  In fact, the corresponding supercuspidal representations should be generic, with wave front set including a principal nilpotent orbit, implying $d_\pi=4$.
\end{example}


One further consequence of Theorem~\ref{T:depth} is a criterion for disjointness of representations of $G_y$ occuring as factors of different Mackey components.

\begin{corollary} \label{C:disjoint}
For $i=1,2$ let $x_i$ be vertices of $\A$ and $\tau_i$ cuspidal representations of $\G_{x_i}$.  Suppose $\ell_i \in X_{x_i,y}^+$ satisfy
$$
\max\{ \alpha(x_1-y+\ell_1) \mid \alpha \in \Phi \} < \max\{ \alpha(x_2-y+\ell_2) \mid \alpha \in \Delta_{x_2} \}.
$$
Then the two Mackey components
$$
\Ind_{G_{[y,x_1+\ell_1]}}^{G_y} \lconj{t(\ell_1)}{\tau_1} \quad \textrm{and} \quad  \Ind_{G_{[y,x_2+\ell_2]}}^{G_y} \lconj{t(\ell_2)}{\tau_2}
$$
are disjoint representations of $G_y$.
\end{corollary}

\section{Case of Deligne-Lusztig cuspidal representations} \label{S:DL}

Our main reference for this section is \cite{Carter1985}.  Recall that a minisotropic (maximal) torus $\T$ of $\G_x = G_{x}/G_{x,+}$ is one which is contained in no proper parabolic subgroup \cite[II.1.11]{SpringerSteinberg1970}.  
Writing $rk(\HH)$ for the $\resk$-rank of the group $\HH$ we set $\varepsilon =(-1)^{rk(\G)-rk(Z(\G_x))}$.

Let  $\T$ be a minisotropic maximal torus of $\G_x$ and $\theta$ a character of $\T$.  From this data P.~Deligne and G.~Lusztig constructed a virtual representation of $\G_x$ whose character we denote $R^{\G_x}_{\T}(\theta)$.
If $\theta$ is in \emph{general position} \cite[\S 7.3]{Carter1985}, then $\varepsilon R^{\G_x}_{\T}(\theta)$ is irreducible and cuspidal, and the corresponding representation $\tau$ is called a Deligne-Lusztig cuspidal representation.  This character is given on an element $h \in \G_x$ with Jordan decomposition $h=su$ by
$\varepsilon R^{\G_x}_{\T}(\theta)(h) = 0$ if $s$ is not conjugate to an element of ${\T}$ and otherwise by
\begin{equation} \label{E:charDL}
\varepsilon R^{\G_x}_{\T}(\theta)(h) = 
\frac{1}{\vert C_{\G_x}^\circ(s) \vert} \sum_{g \in \G_x, gsg^{-1}\in {\T}}\theta(gsg^{-1})Q_{g^{-1}\T g}^{C_{\G_x}^\circ(s)}(u) 
\end{equation}
where $Q_{g^{-1}\T g}^{C_{\G_x}^\circ(s)}$ denotes Green's function, which takes values in $\mathbb{Z}$ \cite[\S 7.6]{Carter1985}.  It is known that
\begin{equation} \label{E:degreeDL}
\deg(\varepsilon R^{\G_x}_{\T}(\theta)) = Q_{\T}^{\G_x}(1) = \frac{\vert \G_x \vert}{\vert \U_x \vert \vert \T \vert}
\end{equation}
where $\U_x$ denotes the unipotent radical of a Borel subgroup $\B_x$ of $\G_x$.

Let us now work towards understanding the Mackey components of the corresponding supercuspidal representation $\pi = \cind_{G_{x}}^G \tau$.  We begin with a general lemma.

\begin{lemma} \label{L:resp}
Let $\tau$ be a depth-zero representation of $G_x$ and $\ell \in X_{x,y}^+$.  Let $F\neq\{x\}$ be the facet of $\A$ which contains $x$ in its closure and meets $[y-\ell,x]$. Let $\Para = G_F/G_{x,+}$ be the parabolic subgroup of $\G_x$ whose inflation to $G_x$ is $G_{F}$.   Then the irreducible components of $\Res_{G_{[y-\ell,x]}}\tau$ coincide with those of $\Res_{\Para}\tau$.
\end{lemma}

\begin{proof}
Since $x \in \overline{F}$ and $F$ is a facet we have $G_{x,+} \subseteq G_{F}$ and $G_F/G_{x,+}$ is indeed a parabolic subgroup of $G_{x}/G_{x,+}$.  Moreover, $G_F/G_{x,+} \cong G_{[y-\ell,x]}/(G_{[y-\ell,x]}\cap G_{x,+})$ since these quotients are uniquely determined by the vanishing of the same affine roots.
So let $\Omega \in \{F, [y-\ell,x]\}$.  
Since $\tau$ is the inflation of a representation, say for the moment $\overline{\tau}$, which is trivial on $G_{x,+}$, $\Res_{G_\Omega/(G_\Omega \cap G_{x,+})}\overline{\tau}$ and $\Res_{G_\Omega}\tau$ have the same irreducible components, and the lemma follows.
\end{proof}

Thus the determination of the decomposition into irreducible subrepresentations of each Mackey component of $\pi$ implies first determining that of the restriction of a cuspidal representation of $\G_x$ to a parabolic subgroup --- a highly nontrivial open problem in general.  Nevertheless, one can deduce some results in an important special case.

Let $\B_x=\Sres\U_x$ be a standard Borel subgroup of $\G_x$.  Then the Jordan decomposition of any $h\in \B_x$ is $h=su$ with $s \in \Sres$ and $u\in \U_x$.  But such an $s$ is conjugate to an element of the minisotropic torus $\T$ if and only if $s \in Z(\G_x)$, since $\T$ cannot contain a split subtorus outside of the center.
Consequently $C_{\G_x}^\circ(s) = \G_x$.   Green's function depends only on the conjugacy class of $\T$ within $C_{\G_x}^\circ(s)$ and so in this case, is simply $Q_{\T}^{\G_x}$.  Since $s$ is central, $gsg^{-1}=s$, and 
the character formula from \eqref{E:charDL} simplifies to
\begin{equation} \label{E:DeligneLusztigrestrict}
\Res_{\B_x} \varepsilon R^{\G_x}_{\T}(\theta)(su) = \begin{cases}
0 & \textrm{if $s \notin Z(\G_x)$},\\
\theta(s)Q_{\T}^{\G_x}(u) & \textrm{otherwise}.
\end{cases}
\end{equation}
An immediate consequence of this calculation is the following lemma.

\begin{lemma}\label{L:borelDL}
The restriction of a Deligne-Lusztig cuspidal representation $\varepsilon R^{\G_x}_{\T}(\theta)$ to a Borel subgroup of $\G_x$ depends only on the choice of minisotropic torus $\T$ (up to conjugacy) and the restriction of $\theta$ to the center $Z(\G_x)$.
\end{lemma}

\begin{remark} \label{R:self}
In general $\Res_{\B_x} \varepsilon R^{\G_x}_{\T}(\theta)$ is not irreducible; in fact its self-intertwining number is
\begin{align*}
\langle \varepsilon R^{\G_x}_{\T}(\theta), \varepsilon R^{\G_x}_{\T}(\theta) \rangle_{\B_x} &=
\frac{1}{\vert {\B_x} \vert} \sum_{s\in Z(\G_x), u \in \U_x}\vert \theta(s) \vert^2 \vert Q_{\T}^{\G_x}(u)\vert^2 \\ 
&= \frac{\vert Z(\G_x) \vert}{\vert {\B_x} \vert} \sum_{u \in \U_x}Q_{\T}^{\G_x}(u)^2.
\end{align*}
\begin{example} 
If $\G_x = \SL(2,\resk)$ then we determined in \cite{Nevins2013} that this intertwining number is $2 = \vert Z(\G_x)\vert$.
\end{example}
\begin{example} \label{Example:SL3self}
If $\G_x = \SL(3,\resk)$ then we compute directly that $\sum_{u \in \U_x}Q_{\T}^{\G_x}(u)^2=q^4(q-1)^2$ and hence that the intertwining number of $\Res_{\B_x} R^{\G_x}_{\T}(\theta)$ with itself is $\vert Z(\G_x) \vert q$, where $\vert Z(\G_x) \vert=3$ if $3$ divides $q-1$.  
\end{example}
\end{remark}

\begin{theorem}\label{T:same}
Let $x,y$ be vertices of $\A$ with $y$ special.  Let $\tau_1$ and $\tau_2$ be two Deligne-Lusztig cuspidal representations of $\G_x$, induced from the same minisotropic torus and with the same central character.  Let $\pi_i=\cind_{G_x}^G\tau_i$ be the corresponding depth-zero supercuspidal representations of $G$, for $i=1,2$.
Then for each $\ell \in int(X_{x,y}^+)$, the Mackey components of $\Res_{G_y}\pi_i$ corresponding to $\ell$ coincide for $i=1,2$.  That is, we have
\begin{equation} \label{E:same}
\Ind_{G_{[y,x+\ell]}}^{G_{y}}\lconj{t(\ell)}{\tau_1} \cong\Ind_{G_{[y,x+\ell]}}^{G_{y}}\lconj{t(\ell)}{\tau_2}.
\end{equation}
\end{theorem}

\begin{proof}
The induced representation
$\Ind_{G_{[y,x+\ell]}}^{G_{y}}\lconj{t(\ell)}{\tau_i}$ is determined by 
$\Res_{G_{[y,x+\ell]}}\lconj{t(\ell)}{\tau_i}$.  Conjugating by $t(\ell)^{-1}$  we deduce that \eqref{E:same} would follow from
\begin{equation} \label{E:ressame}
\Res_{G_{[y-\ell,x]}}\tau_1 \cong \Res_{G_{[y-\ell,x]}}\tau_2.
\end{equation}
By Lemma~\ref{L:alcove}, the geodesic $[y,x+\ell]$ meets a unique alcove $\Gamma'$ adjacent to $x+\ell$; thus $\Gamma=\Gamma'-\ell$ is an alcove adjacent to $x$ meeting $[y-\ell,x]$.  It follows that the group $G_\Gamma/G_{x,+}$ is a Borel subgroup $\B_x$ of $\G_x$ and thus by Lemma~\ref{L:resp}, \eqref{E:ressame} is equivalent to the condition that $\Res_{\B_x}\tau_1 \cong \Res_{\B_x}\tau_2$; this follows from Lemma~\ref{L:borelDL} by our hypotheses on $\tau_1$ and $\tau_2$.
\end{proof}

Similarly, for $\ell \in \partial(X_{x,y}^+)$, the relationship between corresponding Mackey components depends on the restriction of $\tau_i$ to a parabolic subgroup $\Para$.  Since $\T \cap \Para$ is central unless they share a common anisotropic subtorus, these boundary Mackey components will also often coincide.  In an extreme case, we can therefore say much more. 

\begin{corollary} \label{C:allsame}
Suppose we are in the setting of Theorem~\ref{T:same} and suppose additionally  that 
$\T$ has the property that $\T \cap \Para = Z(\G_x)$ for all proper parabolic subgroups $\Para$ of $\G_x$.   
Then if $y$ and $x$ are not conjugate under $G$ we have 
$$
\Res_{G_y}\pi_1 \cong \Res_{G_y}\pi_2,
$$ 
whereas if $y= x$ then there exists a representation $W$ of $G_y$ such that we can write
$$
\Res_{G_y}\pi_i \cong \tau_i \oplus W
$$
for $i=1,2$, with $W$ common to both.
\end{corollary}

\begin{proof}
Under the given hypotheses, the method of the proof of Lemma~\ref{L:borelDL} applies equally to the restriction of $\tau_i$ to any proper parabolic subgroup $\Para$, and we conclude that $\Res_{\Para}\tau_1 \cong \Res_{\Para}\tau_2$.  Therefore, for each $\ell \in X^+_{x,y}$ for which $y-\ell \neq x$, we may apply Lemma~\ref{L:resp} to conclude that $\Res_{G_{[y-\ell,x]}}\tau_1 \cong \Res_{G_{[y-\ell,x]}}\tau_2$ and hence that the corresponding Mackey components coincide, that is, for all such $\ell$, \eqref{E:same} holds.
If $x$ and $y$ lie in distinct orbits under $W$, then $y-\ell\neq x$ holds for all $\ell \in X_{x,y}^+$.  Otherwise, we may without loss of generality assume $y=x$, 
in which case the single non-shared Mackey component is simply
$\Ind_{G_{y}}^{G_{y}}\tau_i = \tau_i$.
\end{proof}

\begin{remark}
The tori satisfying the hypotheses of Corollary~\ref{C:allsame} are rare, as pointed out by the anonymous referee.  It is easy to prove that a minisotropic torus $\T$ satisfies the hypotheses of Corollary~\ref{C:allsame} if and only if each of its $\resk$-rational noncentral subtori are regular.  Recall (see for example \cite{Carter1985}) that one can parametrize all conjugacy classes of $\resk$-tori of $\G_x$ by conjugacy classes in its Weyl group.  Given $w\in W_x^{lin}$ with characteristic polynomial $\chi_w$ and corresponding torus $\T_w$, 
the $\resk$-rational subtori are parametrized by the $w$-invariant sublattices of $X_\ast(S)$, which in turn correspond to the polynomial factors of $\chi_w$ over $\mathbb{Q}$, and in fact $\vert \T_w \vert = \chi_w(q)$. 

 Thus in particular, if $\chi_w$ is irreducible then $\T$ has the desired property.  For classical groups (and also special and general linear groups), one may use the explicit description of these tori given in \cite[\S 1]{Morris1990} to see this condition is also necessary.  For exceptional groups, M.~Reeder has identified many properties of those tori $\T_w$ whose \emph{minimal} polynomial is irreducible; these $w$ he calls \emph{cyclotomic} \cite{Reeder2011}.
\end{remark}

We conclude with two examples.

\begin{example} \label{Example:SL3tori}
If $n$ is prime, the group $\G=\SL(n,\resk)$ has a unique maximal anisotropic torus $\T$ corresponding to the norm one elements of the field extension of $\resk$ of degree $n$, which has order $\Phi_n(q)$, where $\Phi_n$ is the $n$th cyclotomic polynomial.  Thus it has no proper subtori, as one can also deduce from the lack of intermediate extension fields in this case, and hence cannot meet any proper rational Levi subgroup outside $Z(\G)$.
\end{example}

\begin{example}
Let $\G=\mathrm{Sp}(4,\resk)$.  Then $\G$ has two maximal anisotropic tori up to conjugacy \cite[II.4--8]{Springer1970}.  The Coxeter torus $\T_{w_0}$ has order $q^2+1$ and one concludes that it cannot meet a proper parabolic subgroup except in the center of $\G$.  

On the other hand the anisotropic torus $\T_{-1}$ corresponds to the element $w=-1$ in the Weyl group, and has order $(q+1)^2$.  It is  isomorphic to $N_1(\extresk)^2$ where $N_1(\extresk)$ is the group of norm-one elements of a quadratic extension $\extresk$ of $\resk$.   In  \cite{Srinivasan1968} this torus is explicitly described as the subgroup $H_4 = \langle a_4\rangle\times \langle b_4 \rangle$ and one can see directly that neither generator is regular; for example $a_4$ lies in a parabolic subgroup with Levi component isomorphic to $\SL(2,\resk)\times \GL(1,\resk)$.
\end{example}

\section{Intertwining with principal series} \label{S:intertwining}

Let $\chi$ be a depth-zero character of $S$. 
Construct the parabolically induced representation $\Ind_B^G \chi$; this is a depth-zero (possibly reducible) principal series representation of $G$.   We denote by $V$ the space of $\Ind_B^G \chi$. 
For any special vertex $y$,  $G = BG_{y}$ so we have
\begin{equation} \label{E:principalseries}
\Res_{G_{y}}\Ind_{B}^G \chi 
\cong \Ind_{B \cap G_{y}}^{G_{y}} \chi_0
\end{equation}
where $\chi_0$ denotes the restriction of $\chi$ to $B \cap G_y$.  Note that twisting $\chi$ by any unramified character produces the same restriction to $G_y$; this holds in particular for the modular character which appears in our normalized induction  $\Ind_B^G \chi$.

\subsection{Some subrepresentations} \label{SS:filtration}

The nature of parabolic induction is such that it is easier to construct a filtration of $V$ by $G_y$-invariant subspaces than a direct sum decomposition.

\begin{lemma} \label{L:omega}
Let $\Omega$ be a bounded convex closed subset of $\A$ satisfying $\overline{C} \subseteq \Omega \subset \overline{D}$.  Then $\chi_0$ extends trivially to a character of $G_{y+\Omega}$ and $\Ind_{G_{y+\Omega}}^{G_{y}}\chi_0$ is a subrepresentation of $\Ind_{B \cap G_{y}}^{G_{y}} \chi_0$.  Let $V_{y+\Omega}$ denote the space of this representation; it is finite-dimensional.  If $\Omega' \supseteq \Omega$ is another such set, then $V_{y+\Omega'} \supseteq V_{y+\Omega}$.
\end{lemma}

\begin{proof}
By Proposition~\ref{P:stabilizer}, we can write $G_{y+C}$ as $S_0U_{y+C}$ where $S_0$ normalizes $U_{y+C}$.  Since $\chi_0$ is trivial on $S_0 \cap U_{y+C} = S_1$, it extends to a character of $G_{y+C}$, trivial on $U_{y+C}$, which coincides with $\chi_0$ on $B\cap G_y$.  Denote again by $\chi_0$ the restriction of this character to any subgroup of $G_{y+C}$.   Since $\overline{C} \subseteq \Omega \subset \overline{D}$, we have $B\cap G_y \subset G_{y+\Omega} \subseteq G_{y+C}$.  The rest follows.
\end{proof}

We have the following estimates relating to the depth and degree of $V_{y+\Omega}$. 

\begin{proposition} \label{P:depthps}
Suppose $n$ is a positive integer.  If $y+\Omega \subseteq \Omega_y(\A,n)$ then $V_{y+\Omega} \subseteq V^{G_{y,n}}$ and the depth of any irreducible subrepresentation of $V_{y+\Omega}$ is strictly less than $n$.  
Moreover, $\dim(V^{G_{y,n}}) = \vert \G_y / \B \vert\;q^{(n-1)\vert \Phi^+\vert}$ for any Borel subgroup  $\B$ of $\G_y$.
\end{proposition}

Note that it follows that the GK-dimension of $V$ is $\vert \Phi^+\vert$, being half the dimension of the principal nilpotent orbit, as expected.

\begin{proof}
We may restrict to integral $n$ since $y$ is special.  If $y+\Omega \subseteq \Omega_y(\A,n)$ then $G_{y,n} \subseteq G_{y+\Omega}$ by Proposition~\ref{P:gxromega}.  Since $n>0$ we further have $G_{y,n} \subseteq S_1U_{y+\Omega} = \ker(\chi_0)$, so 
it acts trivially on the induced representation, yielding $V_{y+\Omega}^{G_{y,n}} = V_{y+\Omega}$.  In fact this argument defines an isomorphism $V^{G_{y,n}} \cong \Ind_{(B\cap G_y)G_{y,n}}^{G_y}\chi_0$, whence the dimension formula. 
\end{proof}

\begin{remark}
Let $r \in \real_{> 0}$.  If $\Phi$ does not contain an irreducible component of type $G_2$ then from Proposition~\ref{P:gxromega} we may deduce that $(B\cap G_y)G_{y,r}=G_{\Omega_r}$ where $\Omega_r = \overline{D}\cap \Omega_y(\A,r)$. Consequently $V^{G_{y,n}}=V_{\Omega_n}$ in this case.  In general, however, the partially ordered filtration of subrepresentations $V_{y+\Omega}$  does not necessarily include the subrepresentations $V^{G_{y,r}}$ of $G_{y,r}$-fixed vectors.  Although not needed here, note that one can obtain a much finer filtration (which in particular includes the $V^{G_{y,r}}$) by replacing the subgroups $G_{y+\Omega}$ with groups $G_f$ where $f$ is a concave function \cite[\S 6.4]{BruhatTits1972} satisfying $f(\alpha)=-\alpha(y)$ and $f(-\alpha)> \alpha(y)$ for all $\alpha\in \Phi^+$, as in \cite{CampbellNevins2009}. 
\end{remark}

\subsection{Calculations on intertwining}
Now let $\pi = \cind_{G_x}^G \tau$ be a depth-zero supercuspidal representation of $G$.  Let $\ell \in X_{x,y}^+$ and denote as before the corresponding Mackey component  by $\pi_{\ell} = \Ind_{G_{[y,x+\ell]}}^{G_y}\lconj{t(\ell)}{\tau}$.

Then for each set $\Omega$ as in Lemma~\ref{L:omega}, we have
\begin{align} \label{E:chain}
\Hom_{G_y}(\pi_{\ell}, \Ind_{G_{y+\Omega}}^{G_y} \chi_0) &\cong 
\Hom_{G_{[y,x+\ell]}}(\lconj{t(\ell)}{\tau}, \Res_{G_{[y,x+\ell]}}\Ind_{G_{y+\Omega}}^{G_{y}} \chi_0)\\
\notag&\cong \Hom_{G_{[y,x+\ell]}}(\lconj{t(\ell)}{\tau}, \oplus_{c \in \Psi_{x,y,\Omega}} 
\Ind_{G_{[y,x+\ell]} \cap \lconj{c}{G_{y+\Omega}}}^{G_{[y,x+\ell]}} \lconj{c}{\chi_0}) \\
\notag&\cong \oplus_{c \in \Psi_{x,y,\Omega}} \Hom_{G_{[y,x+\ell]} \cap \lconj{c}{G_{y+\Omega}}}(\lconj{t(\ell)}{\tau}, \lconj{c}{\chi_0})\\
\notag&\cong \oplus_{c \in \Psi_{x,y,\Omega}} \Hom_{G_{[x,y-\ell]}\cap G_{t(-\ell)c\cdot (y+\Omega)}}(\tau, \lconj{t(-\ell)c}{\chi_0})
\end{align}
where $\Psi_{x,y,\Omega} = G_{[y,x+\ell]}\backslash G_y / G_{y+\Omega}$.  

Determining a set of representatives for $\Psi_{x,y,\Omega}$ is a large subset of the problem of classifying $B\cap G_y$ double cosets in $G_y$, which for some groups is known to contain the matrix pair problem, that is, be wild \cite{OnnPrasadVaserstein2006}.  Furthermore, while $t(-\ell)c\cdot (y+\Omega)$ will be a convex closed subset of an apartment $\A'$, 
meeting $\A$ in at least the point $y-\ell$, it is not to be expected that there exists a choice of such $\A'$ which also contains $x$.  Thus in general the convex closure of $[x,y-\ell] \cup t(-\ell)c\cdot (y+\Omega)$ is not contained in any apartment of $\B$, and therefore its stabilizer is much more difficult to describe.

Nevertheless, there remain some tractable cases to consider, which suffice for proving the following theorem.
Let $Z_x \subseteq G_x$ denote the full preimage of $Z(\G_x) \subseteq \G_x$.  

\begin{theorem} \label{T:main}
Let $\tau$ be a Deligne-Lusztig cuspidal representation of $\G_x$ with central character $\theta$.  Let $\widehat{\theta}$ denote the inflation of $\theta$ to $Z_x$.
Let $\chi$ be a character of $S$ such that for some $w\in W_0$, $\Res_{Z_x}\lconj{w}{\chi} = \widehat{\theta}$.
Then the restrictions to $G_y$ of
$$
\pi^s=\cind_{G_x}^{G}\tau \quad \textrm{and} \quad \pi^p=\Ind_B^G \chi
$$
have infinitely many distinct irreducible representations in common, of arbitrarily large depth.
\end{theorem}

Note that when $x$ is special, $Z_x=ZS_1$ so the compatibility condition of the theorem is equivalent to simply requiring the representations to have the same central character.

\subsection{Proof of Theorem~\ref{T:main}}

We begin by proving that each Mackey component of a Deligne-Lusztig supercuspidal representation corresponding to an element of $int(X_{x,y}^+)$ intertwines with any compatible principal series representation.

\begin{proposition} \label{P:inttwine}
Let $\tau, \theta$ and $\widehat{\theta}$ be as above.
Let $\ell \in int(X_{x,y}^+)$ and define $w \in \Upsilon_x$ by $\ell+x-y \in wD$.  
Let $\chi$ be a character of $S$ such that $\Res_{Z_x}\lconj{w}{\chi} = \widehat{\theta}$ and denote by $\chi_0$ the trivial extension of $\chi$ to any subgroup of $G_{y+C}$.
Then the representations
$$
 \Ind_{G_{[y,x+\ell]}}^{G_{y}} \lconj{t(\ell)}{\tau} \quad \textrm{and} \quad 
\Ind_{G_{y+\Omega}}^{G_{y}} \chi_0
$$
intertwine, for all bounded convex closed subsets $\Omega$ with $\overline{C} \subseteq \Omega \subset \overline{D}$ for which $x-y+\ell \in w\Omega$.  
\end{proposition}

\begin{proof}
Note that as $S_1 = \lconj{w^{-1}}{S_1} \subseteq Z_x$, the hypotheses imply that $\chi$ has depth zero.  It therefore suffices to show that there exists a nonzero summand in \eqref{E:chain}.  

The existence and uniqueness of $w \in \Upsilon_x \subseteq W_0$ follows from Lemma~\ref{L:upsilon}.  By \eqref{E:wlin}, $w_y := t(y-wy)w \in W_y$, which we lift to an element of $G_y$. 
Set $\Omega' = t(-\ell)w_y\cdot (y+\Omega) = y-\ell+w\Omega$.  When $x-y+\ell \in w\Omega$ both $x$ and $y-\ell$ lie in $\Omega'$, so $G_{[x,y-\ell]}\cap G_{\Omega'} = G_{\Omega'}$.  Defining $U_{\Omega'}$ as in Proposition~\ref{P:stabilizer}, we deduce that the summand corresponding to $c=w_y$ in \eqref{E:chain} is
\begin{equation}\label{E:hom2}
\Hom_{S_0U_{\Omega'}} (\tau, \lconj{t(-\ell)w_y}{\chi_0}).
\end{equation}
By hypothesis we have $x \in int(\Omega')$ so by Corollary~\ref{C:interior}, $U_{\Omega'} \subseteq G_{x,+}\subseteq \ker(\tau)$.  On the other hand, note that
$$
\lconj{w_y^{-1}t(\ell)}{(S_0U_{\Omega'})} = S_0U_{w_y^{-1}t(\ell)\cdot \Omega'} = S_0U_{y + \Omega} 
$$
and that $\chi_0$ was defined to be trivial on $U_{y+\Omega}$.  Therefore $\lconj{t(-\ell)w_y}{\chi_0}$ is trivial on $U_{\Omega'}$.  Moreover, on $S_0$ the character $\lconj{t(-\ell)w_y}{\chi_0}$ coincides with $\lconj{w}{\chi}$.  Thus \eqref{E:hom2} is isomorphic to
$$
\Hom_{S_0} (\tau, \lconj{w}{\chi}).
$$
Using the character formula from \eqref{E:DeligneLusztigrestrict},  the intertwining of the character $\varepsilon R_\T^{\G_x}(\theta)$ of $\tau$ with $\lconj{w}{\chi}$ is given on $S_0$ by
\begin{align*}
\langle \varepsilon R_\T^{\G_x}(\theta),\lconj{w}{\chi} \rangle_{S_0}  
&= \frac{1}{\vert S_0 \vert} \int_{S_0} \varepsilon R_\T^{\G_x}(\theta)(s)\overline{\lconj{w}{\chi}(s)}\; ds \\
&= \frac{1}{\vert S_0 \vert} \int_{Z(\G_x)}\int_{S_1} \deg(\tau)\theta(z)\overline{\lconj{w}{\chi}(zs_1)}\; dzds_1 \\
&= \begin{cases}
0 & \textrm{if $\Res_{Z_x}\lconj{w}{\chi} \neq \widehat{\theta}$}\\
\deg(\tau)\frac{\vert Z(\G_x)\vert}{\vert \Sres\vert} & \textrm{otherwise}.
\end{cases}
\end{align*}
Consequently $\Hom_{S_0} (\tau, \lconj{w}{\chi}) \neq \{0\}$ exactly when the restriction of $\lconj{w}{\chi}$ to $Z_x$ coincides with $\widehat{\theta}$.  The proposition follows.
\end{proof}

We now do away with the apparent dependence on $w$ in Proposition~\ref{P:inttwine}.

\begin{corollary} \label{C:subreps}
Let $\Ind_B^G \chi$ be a depth-zero principal series representation.  Suppose $\tau$ is a Deligne-Lusztig cuspidal representation of $\G_x$ with central character $\theta$ with inflation $\widehat{\theta}$ to $Z_x$.  Let $w\in W_0$ and suppose $\Res_{Z_x}\lconj{w}{\chi} = \widehat{\theta}$.
Then for every $\ell \in int(X_{x,y}^+)$, there exists a subrepresentation of the Mackey component $\pi_\ell$ of $\Res_{G_y} \cind_{G_x}^G\tau$ which is isomorphic to a subrepresentation of $\Res_{G_y}\Ind_B^G \chi$.
\end{corollary}

\begin{proof}
For any $\ell \in int(X_{x,y}^+)$, we define $w_0 \in \Upsilon_x$ as in Proposition~\ref{P:inttwine}.  Thus $x-y+\ell \in w_0D$.  Choose a bounded closed convex set $\Omega$ satisfying 
$C \cup \{w_0^{-1}(x-y+\ell)\} \subset \Omega \subset \overline{D}$.
 Since $\Res_{Z_x}(\lconj{w_0}{(\lconj{w_0^{-1}w}{\chi})}) = \Res_{Z_x}\lconj{w}{\chi} = \widehat{\theta}$,  Proposition~\ref{P:inttwine} implies that
 $\pi_\ell$ intertwines with the subrepresentation of $\Res_{G_y}\Ind_B^G (\lconj{w_0^{-1}w}{\chi})$ induced from $G_{y+\Omega}$.   Consequently $\Res_{G_y}\Ind_B^G (\lconj{w_0^{-1}w}{\chi})$ contains a subrepresentation of $G_y$ which is isomorphic to a subrepresentation of $\pi_\ell$.  Finally, since $w_0^{-1}w \in W_0$,  $\Ind_B^G (\lconj{w_0^{-1}w}{\chi}) \cong \Ind_B^G \chi$ as representations of $G$ and therefore their restrictions to $G_y$ must also be isomorphic.  
\end{proof}

Although the subrepresentations arising in Corollary~\ref{C:subreps} are not necessarily distinct, we have the following result. 

\begin{corollary}
Let $\pi^s$ be a Deligne-Lusztig supercuspidal representation and $\pi^p$ a depth-zero principal series representation, which are compatible in the sense of Corollary~\ref{C:subreps}.  Then  $\Res_{G_y}\pi^s$ and $\Res_{G_y}\pi^p$ have infinitely many distinct components in common, and the set of depths of these components is unbounded.
\end{corollary}

\begin{proof}
The first part follows from Corollary~\ref{C:subreps} by the Pigeonhole Principle since there are infinitely many $\ell \in int(X_{x,y}^+)$ and the admissibility of each supercuspidal representation implies each $G_y$-subrepresentation occurs with finite multiplicity.  

More explicitly, we may restrict $\ell$ to an infinite subset of $X_{x,y}^+ \cap (y-x+D)$ in which every pair of elements satisfy the conditions of Corollary~\ref{C:disjoint} thereby ensuring that their components are distinct.  By Theorem~\ref{T:depth}, the set of depths of these representations is unbounded above.
\end{proof}

\begin{remark}
Given a depth-zero principal series representation, one may ask if for each vertex $x$ and minisotropic maximal torus $\T \subseteq \G_x$ there exists a Deligne-Lusztig cuspidal character $R_\T^{\G_x}(\theta)$ such that the corresponding supercuspidal representation is compatible with $\chi$.  This is equivalent to the question of the existence of a character $\theta$ of $\T$, coinciding with $\chi$ on $Z(\G_x)$, which is in general position, that is, not fixed by any nontrivial element of $W_x$.  For $q$ sufficiently large, this follows from the arguments in \cite[Lemma 8.4.2]{Carter1985} with minor modification.
\end{remark}

\section{An example}\label{S:example}

We now illustrate the use of the results of Sections~\ref{S:restriction} to \ref{S:intertwining} with an example.

Let $G=\SL(3,\ratk)$.  Suppose that $p\neq 3$ and $3 \nmid (q-1)$, whence we have simply  $\GL(3,\R) = Z(\GL(3,\R))\SL(3,\R)$ and the irreducible representations of $\GL(3,\R)$ and $\SL(3,\R)$ coincide up to scalars.  Moreover, all irreducible cuspidal representations of $\SL(3,\resk)$ arise as Deligne-Lusztig cuspidal representations.  The group $GL(3,\ratk)$ acts on $\B=\B(G,\ratk)=\B^{red}(\GL(3),\ratk)$, and this action is transitive on the set of vertices, which are all special.  Note that if $y\in \B$ and $h\in  \GL(3,\ratk)$ we also have $\lconj{h}{G_y} = G_{h\cdot y}$.  It follows that we may without loss of generality fix a single choice of vertex $y$.

Let $\A\subseteq \B$ denote the apartment corresponding to the diagonal split torus, set $\Delta = \{\alpha,\beta\}$, and let $y=0$ be the vertex at which these roots vanish; then $\G_y = \SL(3,\R)$.  The element  $g = \smat{0&0&1\\\p &0&0\\0&\p&0} \in \GL(3,\ratk)$ has the property that $x_0=y, x_1=g\cdot y$ and $x_2=g^2\cdot y$ are the three vertices of the fundamental positive alcove in $\A$, representing the three $\SL(3,\ratk)$-conjugacy classes of vertices in $\B$.  Concretely, we have $\alpha(x_1)=1$ and $\beta(x_1)=0$, so that $x_1 = \frac13(2\alpha^\vee+\beta^\vee)$; $x_2$ is characterized by swapping the roles of $\alpha$ and $\beta$.

For $i\in \{0,1,2\}$, the inflation to $G_{g^i\cdot y}$ of a cuspidal representation of $G_{g^i\cdot y}/G_{g^i\cdot y, +}$ is given by  $\lconj{g^i}{\tau}$, for $\tau$ the inflation of some cuspidal representation of $G_y/G_{y,+}$.
Since $Z(\SL(3,\resk)) = \{1\}$ and there is a unique anisotropic torus in $\SL(3,\resk)$, any two Deligne-Lusztig cuspidal representations of $\SL(3,\resk)$ are compatible, in the sense of Theorem~\ref{T:same}.
Furthermore, as noted in Example~\ref{Example:SL3tori}, the hypotheses of Corollary~\ref{C:allsame} hold for $G$.

Thus for the purposes of exploring all the positive-depth components of the restriction to $\SL(3,\R)$ of all depth-zero supercuspidal representations of $\SL(3,\ratk)$, it suffices to fix an anisotropic torus $\T$ of $\SL(3,\resk)$ and a 
Deligne-Lusztig cuspidal representation $\tau = R_{\T}^{\G}(\theta)$, and 
 consider for $i\in \{0,1,2\}$ the corresponding supercuspidal representations 
$$\pi^i= \cind_{G_{g^i\cdot y}}^G \lconj{g^i}{\tau}.$$
We will denote the various Mackey components of $\Res_{G_y}\pi^i$ by  $\pi^i_\ell$.

When $i=0$, $X_{x_0,y}^+ = X_+$.   The Mackey component corresponding to $\ell=0$ in $\pi^0$ is simply $\pi_0^0=\tau$, which has depth zero. Applying Theorem~\ref{T:depth}, we deduce that for $\ell \in X_+ \setminus\{0,\alpha^\vee+\beta^\vee\}$, the depth of any irreducible subrepresentation of $\pi^0_\ell$ is at least $2$.  Furthermore, when $\ell = \alpha^\vee+\beta^\vee$, the depths of the irreducible subrepresentations of $\pi^0_\ell$ are confined between $1$ and $2$.

When $i\in \{1,2\}$, deciphering the definition given in Proposition~\ref{P:doublecosetreps} yields 
$X_{x_1,y}^+ = X_+ \cup \{\beta^\vee + 3nx_2 \mid n\geq 0\}$ and $X_{x_2,y}^+ = X_+ \cup \{\alpha^\vee + 3nx_1 \mid n\geq 0\}$.  Theorem~\ref{T:depth} implies that the irreducible subrepresentations of $\pi^i_0$ all have depth exactly $1$, and for any other $\ell \in X_{x_i,y}^+$, the depth of any irreducible subrepresentation of $\pi^i_\ell$ is at least $2$.   

In the following paragraphs, we determine the decomposition into irreducible subrepresentations of $\pi^i_\ell$, for $(i,\ell)\in \{(0,\alpha^\vee+\beta^\vee), (1,0), (2,0)\}$.  We obtain three distinct irreducible depth-one representations of $\SL(3,\R)$; we show they occur among (but do not exhaust) the depth-one irreducible representations occuring in the branching rules of the unramified principal series.  

Let $G_{abc}$, for $0\leq a,b \leq c \leq a+b$, denote the subgroup which is the intersection with $G$ of the set of matrices of the form $\smat{\R & \R & \R \\ \PP^a & \R & \R \\ \PP^c & \PP^b & \R}$.  Then following \cite{CampbellNevins2009}, $\Ind_{G_{abc}}^{G_y}1$ is a $G_y$-subrepresentation of the unramified principal series, which we denote $V_{abc}$.  The quotient of $V_{abc}$ by all the $V_{a'b'c'}$ it properly contains is denoted $W_{abc}$, which may or may not be irreducible; its decomposition into irreducible subrepresentations is given in \cite{CampbellNevins2009} and \cite{OnnSingla2014}.  Finally, recall that Green's function $Q_\T^\G$ is given by \cite{Green1955} $Q_\T^\G(1)= (q-1)(q^2-1) = \deg(\tau)$ and
$$
Q_{\T}^{\G}(u) = \begin{cases}
1-q & \textrm{if $\textrm{rank}(u-1) =1$;}\\
1 & \textrm{if $\textrm{rank}(u-1) = 2$.}
\end{cases}
$$

First let $(i,\ell)=(1,0)$; note that $0\in \partial(X_{x_1,y}^+)$.  We compute that $G_{[0,x_1]} = G_{101}$, which has index $q^2+q+1$ in $G_0$.
We deduce from Proposition~\ref{P:degree} that the degree of $\pi^1_0$ is $\deg(\lconj{g}{\tau})[G_0:G_{101}]=(q^2-1)(q^3-1)$.   

Using the character formula \eqref{E:charDL} for $\tau$, one can check that 
$\dim(\Hom_{G_{101}}(\lconj{g}{\tau},\lconj{g}{\tau}))=1$.  
To show further that $\pi^1_0=\Ind_{G_{101}}^{G_y}\lconj{g}{\tau}$ is irreducible,  we consider
$$
\Hom_{G_y}(\pi^1_0,\pi^1_0)\cong \oplus_{\gamma \in G_{101}\backslash G_y/G_{101}}\Hom_{G_{101}\cap\lconj{\gamma}{G_{101}}}(\lconj{g}{\tau},\lconj{\gamma g}{\tau}).
$$
Lists of representatives for such double cosets are given in \cite{CampbellNevins2009}; in this case there are two.  Applying character formulas we may again compute directly that this space of self-intertwining operators
is one-dimensional.  

Now let us consider the intertwining of $\pi^1_0$ with a principal series representation.  Since $Z(\G_0)=\{1\}$, the compatibility condition of Theorem~\ref{T:main} holds for any depth zero character of $S_0$, so without loss of generality let $\chi=1$ and consider the unramified principal series $V=\Ind_B^G 1$.  Using \eqref{E:charDL} we can determine that 
$$
\dim(\Hom_{G_{101}\cap G_{212}}(\lconj{g}{\tau},1))=1
$$
and deduce that $\pi^1_0$ intertwines nontrivially with $V_{212} = \Ind_{G_{212}}^{G_0} 1$.  By dimension arguments, or else a direct computation to rule out intertwining with any subquotient $V_{abc}$ of $V_{212}$, we conclude that $\pi^1_0 \cong W_{212}$, which is irreducible \cite{CampbellNevins2009}.  
A symmetric analysis yields $\pi^2_0 \cong W_{122}$, which is distinct from $W_{212}$.

The case for $(i,\ell)=(0,\alpha^\vee+\beta^\vee)$ is more interesting.
Note that $G_{[0,x_0+\ell]}=G_{112}$.  
As above, we determine that $\deg(\pi^0_\ell)=q(q+1)(q^2-1)(q^3-1)$ using Proposition~\ref{P:degree}.

By Remark~\ref{R:self} and as computed in Example~\ref{Example:SL3self}, the intertwining number of $\Res_{\B^{op}}\tau$ with itself is $q$.  To decompose $\Res_{\B^{op}}\tau$ into irreducible subrepresentations, we begin by restricting $\tau$ to the unipotent radical $\U^{op}$ of $\B^{op}$,  which is simply a Heisenberg group over $\resk$ with center $\GG_{-\alpha-\beta}(\resk)$.  Using character computations, one determines that the restriction of $\tau$ to $\U^{op}$ consists of $(q-1)$ copies of each of the $(q-1)$ distinct Stone-Von Neumann representations $H_{\psi}$ (corresponding to the nontrivial characters $\psi$ of $\resk$) together with $(q-1)^2$ characters of $\U^{op}$.  (These characters correspond to the characters $\psi_{-\alpha}\otimes \psi_{-\beta}\otimes 1$ of $\GG_{-\alpha}(\resk)\times \GG_{-\beta}(\resk)\times \GG_{-\alpha-\beta}(\resk)$ where neither $\psi_{-\alpha}$ nor $\psi_{-\beta}$ is trivial.)  

It is then straightforward to determine that $\Res_{\B^{op}}\tau$ decomposes as $q$ distinct irreducible representations: for each of the $q-1$ nontrivial characters $\psi$, the representation $\rho_{\psi} = \Ind_{\U^{op}}^{\B^{op}} H_{\psi}$ of degree $q(q-1)$;  and the representation $\rho_1 = \Ind_{\U^{op}}^{\B^{op}} \psi_{-\alpha}\otimes \psi_{-\beta}\otimes 1$, of degree $(q-1)^2$, where $\psi_{-\alpha}$ and $\psi_{-\beta}$ are any pair of nontrivial characters.

We obtain a corresponding decomposition 
$\pi^0_\ell = \oplus_{\psi}\rho_{\psi}'$, where $\rho_{\psi}':=\Ind_{G_{112}}^{G_y}\lconj{t(\ell)}{\rho_\psi}$.  One can show that each $\rho_{\psi}'$ is irreducible by computing directly that of the seven double cosets of $G_{112}$ in $G_y$, only the trivial one supports nonzero intertwining operators.

Now let us consider the intertwining of $\pi^0_\ell$ with a principal series representation.  Set ${\Omega_3} = \overline{D} \cap \Omega_0(\A,3)$.  By Proposition~\ref{P:depthps} and the remarks following, $V_{\Omega_3}=V^{G_{0,3}}$.  Since $V_{\Omega_3}=V_{333}$, and since all $G_y$-subrepresentations of $V$ of depth at most $2$ lie in $V^{G_{0,3}}$, we deduce that all intertwining of $\pi^0_\ell$ with $V$ occurs with the $q^6(q^2+q+1)(q+1)$-dimensional subrepresentation $V_{333}$.

By the proof of Proposition~\ref{P:inttwine}, the intertwining number of $\pi^0_\ell$ with $V_{333}$ is at least $\deg(\tau)/\vert \Sres \vert = q+1$, which suggests the possibility that $\pi^0_\ell$ can be embedded into $V$ as a subrepresentation.  This is in fact the case, as follows.

Using the same arguments as in the proof of Proposition~\ref{P:inttwine}, one can show directly that the irreducible representation $\rho_1'$ intertwines already with $V_{222}$; in particular this implies that $\rho_1'$ has depth $1$, which follows readily from its construction.  Furthermore, this intertwining occurs for no larger subgroup than $G_{222}$, whence $\rho_1'$ lives in the highest-dimensional quotient $W_{222}$.  By \cite{CampbellNevins2009}, this component is irreducible of dimension $q(q^2-1)(q^3-1)$ (whence another proof of the irreducibility of $\rho_1'$) and occurs in $V$ with multiplicity $1$.

This leaves $q$ independent intertwining maps of $\rho'=\oplus_{\psi \neq 1}\rho_{\psi}'$ with $V_{333}$.  By computing the intertwining of $\rho'$ with the subrepresentations $V_{223}$ and $V_{222}$ of $V_{333}$, we isolate their image to  $V_{223}/V_{222} = W_{123}\oplus W_{223} \oplus W_{213}$.  In \cite{CampbellNevins2009} it is shown that $W_{123}\cong W_{213}$ are isomorphic and irreducible of dimension $q^2(q+1)(q^3-1)$ and in \cite{OnnSingla2014} it is shown that $W_{223}$ decomposes into a direct sum of $q-2$ distinct irreducible representations of this same degree, each distinct from $W_{123}$.  
We deduce that $\rho'$ coincides with $W_{123}\oplus W_{223} \cong W_{213}\oplus W_{223}\subseteq V$.  

Consequently (\emph{cf.} Corollary~\ref{C:subreps}) $\pi^0_\ell$ embeds in $V$ (nonuniquely!); in fact, we have the stronger result that
$$
\pi^0_\ell \oplus \pi^1_0 \oplus \pi^2_0 \cong (W_{222}\oplus W_{123}\oplus W_{223})\oplus W_{212} \oplus W_{122} \subseteq V.
$$
However, $V$ contains additional $G_y$-subrepresentations of depth $1$, namely $W_{202}\cong W_{112} \cong W_{022}$, which by our analysis do not occur in the restriction of any depth-zero supercuspidal representation of $G$.

\providecommand{\bysame}{\leavevmode\hbox to3em{\hrulefill}\thinspace}
\providecommand{\MR}{\relax\ifhmode\unskip\space\fi MR }
\providecommand{\MRhref}[2]{%
  \href{http://www.ams.org/mathscinet-getitem?mr=#1}{#2}
}
\providecommand{\href}[2]{#2}

\end{document}